\documentclass[a4paper,11pt]{amsart}
\usepackage{amsfonts,amssymb,amsmath,amsthm,abstract}
\usepackage[ps,all,arc,rotate]{xy}
\usepackage[lmargin=1in,rmargin=1in,tmargin=1in,bmargin=1in]{geometry}
\usepackage{fancyhdr}
\usepackage{pb-diagram}

\newtheorem{prop}{Proposition}[section]
\newtheorem{lemma}[prop]{Lemma}
\newtheorem{thm}[prop]{Theorem}
\newtheorem{cor}[prop]{Corollary}

\theoremstyle{definition}
\newtheorem{defn}[prop]{Definition}

\newtheorem{rmk}[prop]{Remark}
\newtheorem{ex}[prop]{Example}

\DeclareMathOperator{\proj}{Proj} \DeclareMathOperator{\quot}{Quot} \DeclareMathOperator{\rep}{Rep}
\DeclareMathOperator{\red}{red}   \DeclareMathOperator{\pur}{pure}  \DeclareMathOperator{\Stab}{Stab}      \DeclareMathOperator{\reg}{reg}
\DeclareMathOperator{\diag}{diag}
\newcommand{\ra}{\rightarrow}      \newcommand{\onto}{\twoheadrightarrow}

\def\cB{\mathcal B}\def\cC{\mathcal C}
\def\cE{\mathcal E}\def\cF{\mathcal F}\def\cH{\mathcal H}
\def\cI{\mathcal I}
\def\cO{\mathcal O}
\def\cR{\mathcal R}\def\cS{\mathcal S}
\def\cU{\mathcal U}

\def\AA{\mathbb A}
\def\GG{\mathbb G}

\def\NN{\mathbb N}\def\PP{\mathbb P}
\def\QQ{\mathbb Q}\def\RR{\mathbb R}

\def\ZZ{\mathbb Z}

 \def\GL{\mathrm{GL}} \def\SL{\mathrm{SL}}

\title{Stratifications for moduli of sheaves and moduli of quiver representations}

\author{Victoria Hoskins}

\begin{document}

\maketitle

\begin{abstract}
We study the relationship between two stratifications on parameter spaces for coherent sheaves and for quiver representations: a stratification by Harder--Narasimhan types and a stratification arising from the geometric invariant theory construction of the associated moduli spaces of semistable objects. For quiver representations, both stratifications coincide, but this is not quite true for sheaves. We explain why the stratifications on various Quot schemes do not coincide and see that the correct parameter space to compare such stratifications is the stack of coherent sheaves. Then we relate these stratifications for sheaves and quiver representations using a generalisation of a construction of \'{A}lvarez-C\'{o}nsul and King.
\end{abstract}

\section{Introduction}

Many moduli spaces can be described as a 
quotient of a reductive group $G$ acting on a scheme $B$ using the methods of  
geometric invariant theory (GIT) developed by Mumford \cite{mumford}. This 
depends on a choice of linearisation of the $G$-action which determines 
an open subscheme $B^{ss}$ of $B$ of semistable points such that the 
GIT quotient $B/\!/ G$ is a good quotient of $B^{ss}$. In this article, 
we are interested in two moduli problems with such a GIT construction:
\begin{enumerate}
\item moduli of representations of a quiver $Q$ with relations $\cR$;
\item moduli of coherent sheaves on a polarised projective scheme $(X,\cO_X(1))$.
\end{enumerate} 
In the first case, King \cite{king} uses a stability parameter $\theta$ to 
construct moduli spaces of $\theta$-semistable quiver representations of dimension 
vector $d$ as a GIT quotient of an affine variety by a reductive 
group action linearised by a character $\chi_\theta$. In the second case, 
Simpson \cite{simpson} constructs a moduli space of semistable sheaves with 
Hilbert polynomial $P$ as a GIT quotient of a projective scheme by a reductive 
group action linearised by an ample line bundle.

For both moduli problems, we compare two stratifications: a Hesselink 
stratification arising from the GIT construction and a stratification by 
Harder--Narasimhan types. For quiver representations, both stratifications 
coincide, but this is not the case for sheaves. We explain why these 
stratifications do not agree for sheaves and how to rectify this. 
Then we relate these stratifications for sheaves and quiver representations using a 
construction of \'{A}lvarez-C\'{o}nsul and King \cite{ack}.

\subsection{Harder--Narasimhan stratifications}

In both of the above moduli problems, there is a notion of semistability for 
objects that involves verifying an inequality for all subobjects; 
in fact, this arises from the GIT notion of semistability 
appearing in the construction of these moduli spaces. 
The idea of a Harder--Narasimhan (HN) filtration is to construct 
a unique maximally destabilising filtration for each object in a moduli 
problem \cite{hn}.

Every coherent sheaf has a unique HN filtration: for pure sheaves, this result 
is well-known (cf. \cite{huybrechts} Theorem 1.3.4) and the extension to coherent 
sheaves is due to Rudakov \cite{rudakov}. 
For quiver representations, there is no canonical notion of HN 
filtration with respect to the stability parameter $\theta$. Rather, the 
notion of HN filtration depends on a collection of positive integers 
$\alpha_v$ indexed by the vertices $v$ of the quiver (see Definition 
\ref{defn HN}). We note that in many previous works, 
only the choice $\alpha_v =1$, for all vertices $v$, is considered.

For both quiver representations and sheaves, we can stratify the 
associated moduli stacks by HN types (which encode the 
invariants of the successive quotients in the HN filtrations) and 
we can stratify the parameter spaces used in the GIT 
construction of these moduli spaces.

\subsection{Hesselink stratifications}

Let $G$ be a reductive group acting on a scheme $B$ such that
\begin{enumerate}
\item $B$ is affine and the action on the structure sheaf is 
       linearised by a character of $G$ , or
\item $B$ is projective with an ample $G$-linearisation.
\end{enumerate}
Then, by the Hilbert--Mumford criterion \cite{mumford}, it suffices 
to check GIT semistability on 1-parameter subgroups (1-PSs) 
of $G$. Associated to this action and a choice of norm on 
the set of conjugacy classes of 1-PSs, there is a stratification 
due to Hesselink \cite{hesselink} of $B$
\[ B = \bigsqcup_{\beta \in \cI} S_\beta\]
into finitely many disjoint $G$-invariant locally closed subschemes. 
Furthermore, there is a partial order on the strata such that the 
closure of a stratum is contained in the union over all higher strata, the 
lowest stratum is $B^{ss}$ and the higher strata parametrise points of 
different instability types. In order to construct this stratification, 
the idea is to associate to each unstable point a conjugacy class of 
adapted 1-PSs of $G$, in the sense of Kempf \cite{kempf}, 
that are most responsible for the instability of this point. 

These stratifications have a combinatorial nature: 
the index set can be determined from the weights of the action of 
a maximal torus of $G$ and the strata $S_\beta$ can be constructed from 
simpler subschemes $Z_\beta$, known as limit sets, which are semistable 
sets for smaller reductive group actions (cf. \cite{kirwan,ness} for the 
projective case and \cite{hoskins_quivers} for the affine case).

Hesselink stratifications have a diverse range of applications. 
For a smooth projective variety $B$, the cohomology of $B/\!/G$ 
can be studied via these stratifications \cite{kirwan}. 
In variation of GIT, these stratifications are used to describe the 
birational transformations between quotients \cite{dolgachevhu,thaddeus} 
and to provide orthogonal decompositions in the derived 
categories of GIT quotients \cite{bfk,halpern}.

\subsection{Comparison results}

For quiver representations, the choice of parameter $\alpha$ 
used to define  HN filtrations corresponds to a choice of 
norm $|| - ||_{\alpha}$ used to define the Hesselink stratification.
The corresponding HN and Hesselink stratifications on the 
space of quiver representations agree: for a quiver without 
relations, this result is \cite{hoskins_quivers} Theorem 5.5 and we deduce the 
corresponding result for a quiver with relations in Theorem \ref{quiver HN is Hess}.

Simpson constructs the moduli space of semistable sheaves on $(X,\cO_X(1))$ 
with Hilbert polynomial $P$ as a GIT quotient of a closed subscheme $R_n$ of the Quot scheme 
\[\quot_n:=\quot(k^{P(n)} \otimes \cO_X(-n),P)\] by the natural $\SL_{P(n)}$-action 
for $n$ sufficiently large. The associated Hesselink stratification of $R_n$ 
was partially described in \cite{hoskinskirwan} and compared with the stratification 
by HN types. It is natural to expect these stratifications to coincide, as in the 
case of quiver representations, but it is only shown in \cite{hoskinskirwan} 
that a pure HN stratum is contained in a Hesselink stratum for $n$ sufficiently large.

In this article, we complete the description of the Hesselink stratification of $\quot_n$ 
and extend the inclusion result of \cite{hoskinskirwan} to non-pure HN types (cf. 
Theorem \ref{thm comp strat}). Moreover, we explain why the Hesselink and HN stratifications on the Quot scheme do not agree. 
The underlying moral reason is that the Quot scheme 
does not parametrise all sheaves on $X$ with all Hilbert polynomial $P$ 
and so is only a truncated parameter space. 
The natural solution is to instead consider the moduli stack 
$\cC oh_P(X)$ of coherent sheaves on $X$ with Hilbert polynomial $P$. 
As every coherent sheaf is $n$-regular for $n$ sufficiently large, 
the open substacks of $n$-regular sheaves form an 
increasing cover
\[ \cC oh_P(X) = \bigcup_{n \geq 0} \cC oh_P^{n-\reg}(X)\]
such that $\cC oh^{n-\reg}_P(X) $ is a stack quotient of an open 
subscheme of $\quot_n$ by $\GL_{P(n)}$. Using this description, we can view the 
Quot schemes as finite dimensional approximations to $\cC oh_P(X)$. 
We use the Hesselink 
stratifications on each $\quot_n$ to construct an associated Hesselink stratification 
on $\cC oh_P(X)$, which can be seen as a limit over $n$ of the stratifications 
on $\quot_n$. The main result is that the Hesselink and HN stratifications on $\cC oh_P(X)$ coincide (cf. \ref{cor Hess is HN on stack}).

\subsection{The functorial construction of \'{A}lvarez-C\'{o}nsul and King}

We relate these stratifications for sheaves and quivers using a 
functor considered by \'{A}lvarez-C\'{o}nsul and King \cite{ack}
\[ \Phi_{n,m}:= \text{Hom} (\cO_X(-n) \oplus \cO_X(-m),\:-\:) : 
\textbf{Coh}(X) \ra \textbf{Reps}({K_{n,m}})\]
from the category of coherent sheaves on $X$ to the category of representations 
of a Kronecker quiver $K_{n,m}$ with two vertices ${n,m}$ and 
$\dim H^0(\cO_X(m-n))$ arrows from $n$ to $m$. Explicitly, a sheaf $\cE$ is 
sent to the $K_{n,m}$-representation $W_\cE=(W_{\cE,n}, W_{\cE,m}, \text{ev}_\cE)$ where 
$W_{\cE,l}:= H^0(\cE(l))$ and the morphisms are given by the evaluation map
\[ \text{ev}_\cE:  H^0(\cE(n)) \otimes H^0(\cO_X(m-n)) \ra H^0(\cE(m)).\]
\'{A}lvarez-C\'{o}nsul and King \cite{ack} show, 
for $m >\!> n > \!> 0$, this functor embeds the subcategory of semistable sheaves 
with Hilbert polynomial $P$ into the subcategory of $\theta_{n,m}(P)$-semistable 
quiver representations of dimension $d_{n,m}(P)$ and construct the moduli space 
of semistable sheaves on $X$ with Hilbert polynomial $P$ by using King's 
construction of the moduli spaces of quiver representations. 

In the final section of this article, we relate the stratifications for moduli 
of quivers and sheaves using this map of \'{A}lvarez-C\'{o}nsul and King. 
We prove that a HN stratum for sheaves is sent to a HN stratum for quiver 
representations (cf. Theorem \ref{thm ack HN}); although, multiple HN strata for sheaves 
can be sent to the same HN stratum for quiver representations when $\dim X > 1$.

\subsection*{Notation and conventions}

We work over an algebraically closed field $k$ of characteristic zero and 
by scheme, we mean scheme of finite type over $k$. By sheaf, we always mean 
coherent algebraic sheaf. For schemes $X$ and $S$, by a family of sheaves 
on $X$ parametrised by $S$, we mean a sheaf $\cF$ over $X \times S$ 
that is flat over $S$ and we write $\cF_s := \cF|_{X \times \{s\}}$. 

We use the term stratification in a weaker sense than usual to mean a 
decomposition into disjoint locally closed subsets with a partial order on 
the strata such that the closure of a given stratum is contained in the 
union of all higher strata (usually for a stratification, one requires 
the closure of a given stratum to be the union of all higher strata).

For natural numbers 
$m$ and $n$, we write \lq for $m >\!> n$' or \lq for $m$ sufficiently larger 
than $n$' to mean there exists $N \geq n$ such that for all $m \geq N$. Similarly, by \lq for $n_r >\!> n_{r-1} >\!> \cdots >\!> n_0$', we mean $\exists N_0 \geq n_0 \: \forall n_1 \geq N_0 \: \exists N_1 \geq n_1 \: \forall n_2 \geq N_1 \dots \: \exists N_{r-1} \geq n_{r-1} \: \forall n_r \geq N_{r-1}$.

\section{Hesselink stratifications of affine varieties}\label{sec hess aff}

Let $G$ be a reductive group acting linearly on an affine variety $V$ and 
let $L_\rho$ denote the $G$-linearisation on the structure sheaf of $V$ obtained 
by twisting by a character $\rho : G \ra \GG_m$. 

\begin{defn}\label{defn GIT ss wrt char} A point $v \in V$ is $\rho$-{semistable} 
if there exists $\sigma \in H^0(V, {L}_\rho^{\otimes n})^G $ for some $n > 0$ such that $\sigma(v) \neq 0$. 
 We let $V^{\rho-\text{ss}}$ denote the subset of $\rho$-semistable points.
\end{defn}

By \cite{mumford} Theorem 1.10, the GIT quotient of $G$ acting on $V$ with respect to $L_\rho$
\[ V^{\rho-\text{ss}} \ra V/\!/_\rho G := \proj \bigoplus_{n \geq 0} H^0(V, {L}_\rho^{\otimes n})^G \]
is a good quotient. Let $(-,-) : X^*(G) \times X_*(G) \ra \ZZ$ be the natural pairing between 
characters and cocharacters; then semistability can be checked on 
1-parameter subgroups $\lambda : \GG_m \ra G$.

\begin{thm}[The Hilbert--Mumford criterion \cite{king,mumford}]
A point $v \in V$ is $\rho$-semistable if and only we have $(\rho,\lambda) \geq 0$, 
for every 1-PS $\lambda : \GG_m \ra G$ 
such that $\lim_{t \ra 0 } \lambda(t) \cdot v$ exists.
\end{thm}

The unstable locus $V - V^{\rho-ss}$ can be stratified by instability types by 
assigning to each unstable point a conjugacy class of 1-PSs 
that is \lq most responsible' for the instability of this point and then stratifying 
by these conjugacy classes. To give a precise meaning to the term 
\lq most responsible' we use a norm $|| - ||$ on the set $\overline{X}_*(G)$ of conjugacy classes of 1-PSs of $G$. More precisely, we fix a maximal torus $T$ of $G$ and a 
Weyl invariant norm $|| - ||_T$ on $X_*(T)_\RR$; then, for $\lambda \in X_*(G)$, we let $|| \lambda ||:= || g\lambda g^{-1}||_T$ 
for $g \in G$ such that $g\lambda g^{-1} \in X_*(T)$.

\begin{ex}\label{ex norms} (a)
Let $T$ be the diagonal maximal torus in $G = \GL(n)$;  
then the norm associated to the dot product on $\RR^n \cong X_*(T)_\RR $ is invariant under the Weyl group $W = S_n$.

(b) For a product of general linear groups $G=\GL(n_1)\times\cdots \times\GL(n_r)$, 
we can use positive numbers $\alpha \in \NN^r$ to weight the norms $|| - ||$ for each factor constructed by part (a); 
that is, 
\[ || (\lambda_1,\dots , \lambda_r) ||_\alpha := \sum_{i=1}^r \alpha_i ||\lambda_i||.\]
\end{ex}

We fix a norm $|| - ||$ on $\overline{X}_*(G)$ such that $|| -||^2$ is $\ZZ$-valued. In 
the study of instability in projective GIT, Kempf used such a norm to define 
a notion of adapted 1-PSs \cite{kempf}. 
In the affine setting, we use an appropriately modified definition as given in 
\cite{hoskins_quivers}. For $v \in V$, we let 
\[ M^\rho_G(v) := \inf\left\{\frac{(\rho, \lambda)}{|| \lambda ||}:\:\text{1-PSs}\: 
\lambda \text{ of } G \text{ such that } \lim_{t \to 0} \lambda(t) \cdot v\: 
\text{exists}\right\}.\]

\begin{defn}
A 1-PS $\lambda$ is $\rho$-adapted to a $\rho$-unstable point $v$ if $\lim_{t \to 0} \lambda(t) \cdot v$ exists and
\[ M^\rho_G(v) = \frac{(\rho,\lambda)}{|| \lambda ||}. \]
We let $\wedge^\rho(v) $ denote the set of primitive 1-PSs which are $\rho$-adapted to 
$v$. Let $[\lambda]$ denote the conjugacy class of a 1-PS $\lambda$; then the Hesselink stratum associated to $[\lambda]$ is
\[S_{[\lambda]}:=\{ v \in V - V^{\rho\text{-}\mathrm{ss}}: 
\wedge^\rho(v) \cap [\lambda] \neq \emptyset\}.\]
We define a strict partial ordering $<$ on $\overline{X}_*(G)$ by $[\lambda] < [\lambda']$ if 
\[ \frac{(\rho,\lambda)}{||\lambda||} > \frac{(\rho,\lambda')}{||\lambda'||}. \]
\end{defn}

Let $V^\lambda_+ $ be the closed subvariety of $V$ consisting of points $v$ such that 
$\lim_{t \to 0} \lambda(t) \cdot v$ exists; then we have a natural retraction 
$p_\lambda : V^\lambda_+ \ra V^\lambda$ onto the $\lambda$-fixed locus. The following theorem 
describing the Hesselink strata appears in \cite{hoskins_quivers} for the case when 
$V$ is an affine space.

\begin{thm}\label{hess strat thm}
Let $G$ be a reductive group acting linearly on an affine variety $V$ with 
respect to a character $\rho$. Let $|| - ||$ be a norm on $\overline{X}_*(G)$; then there is a decomposition 
\[ V -V^{\rho\text{-}\mathrm{ss}} = \bigsqcup_{[\lambda] \in \cB} S_{[\lambda]} \]
into finitely many disjoint $G$-invariant locally closed subvarieties of $V$. 
Moreover, we have
\begin{enumerate}
\item $S_{[\lambda]}= GS_\lambda$ where 
$S_{\lambda}:= \{ v:  \lambda \in \wedge^\rho(v)\}$,
\item $S_\lambda = p_\lambda^{-1}(Z_\lambda) $ where 
$ Z_{\lambda} := \{ v :  \lambda \in \wedge^\rho(v) \cap G_v\} $,
\item $Z_{\lambda}$ is the semistable subset for 
$G_\lambda:= \Stab_G(\lambda)$ acting on $V^{\lambda}$ with respect to the 
character $\rho_\lambda:= || \lambda||^2 \rho - (\rho,\lambda)\lambda^*$, 
where $\lambda^*$ is the $|| - ||$-dual character to $\lambda$, 
\item $\partial S_{[\lambda]} \cap S_{[\lambda']} \neq \emptyset$ only if 
$[\lambda] < [\lambda']$.
\end{enumerate}
\end{thm}
\begin{proof} We deduce the result for a closed $G$-invariant subvariety $V$ of an 
affine space $W$ from the result for $W$ given in \cite{hoskins_quivers}. 
Since taking invariants for a reductive group $G$ is exact 
and $V \subset W$ is closed, it follows that $V^{\rho-ss} = V \cap W^{\rho-ss}$ and  
the Hesselink stratification of $V- V^{\rho-ss}$ is the intersection of 
$V- V^{\rho-ss}$ with the Hesselink stratification of $W - W^{\rho-ss}$; that is,
\[    V = \bigsqcup_{[\lambda]} S_{[\lambda]}^V \quad \text{and} \quad 
      W = \bigsqcup_{[\lambda]} S_{[\lambda]}^W, \quad \quad 
   \text{where} \quad S_{[\lambda]}^V =V \cap S_{[\lambda]}^W .\]
We use properties of the strata for $W$ to prove the analogous properties of 
the strata for $V$.
\begin{enumerate}
\item As $V$ is $G$-invariant, we have that 
$S_{[\lambda]}^V : =V \cap S_{[\lambda]}^W = G(V \cap S_\lambda^W)=G S_\lambda^V$.
\item As $V\subset W$ is closed, we have that $S_\lambda^V=V \cap S_\lambda^W=
p_\lambda^{-1}(V \cap Z_\lambda^W)=p_\lambda^{-1}(Z_\lambda^V)$.
\item As $G_\lambda$ is reductive and $V \subset W$ is closed, we have that 
$Z_\lambda^V = V \cap Z_\lambda^W$ is the GIT semistable locus for $G_\lambda$ 
acting on $V^\lambda= V \cap W^\lambda$.
\item As $V \subset W$ is closed, we have that $\emptyset \neq \partial S_{[\lambda]}^V 
\cap S_{[\lambda']}^V \subset \partial S_{[\lambda]}^W \cap S_{[\lambda']}^W $ 
only if $[\lambda] < [\lambda']$.
\end{enumerate}
This completes the proof of the theorem.
\end{proof}

By setting $S_{[0]}:= V^{\rho\text{-ss}}$, we obtain a Hesselink stratification of $V$
\begin{equation}\label{hess strat} V = \bigsqcup_{[\lambda]} S_{[\lambda]}. \end{equation}

\subsection{Computing the Hesselink stratification}\label{sec comp Hess strat}

The task of computing these stratifications is greatly simplified by 
Theorem \ref{hess strat thm}, as we can construct the strata from the limit sets 
$Z_\lambda$. Furthermore, we will see that, by fixing a maximal torus $T$ 
of $G$, we can determine the indices for the unstable strata from the $T$-weights of the 
action on $V$.

\begin{defn} For $v \in V$, we let $W_v$ denote the set of $T$-weights of $v$. 
For each subset $W$ of $T$-weights, we consider an associated cone in 
$ X_*(T)_\RR = X_*(T) \otimes_\ZZ \RR$
\[ C_W := \bigcap_{\chi \in W} H_\chi  \quad \quad \text{where} \: \: 
H_\chi:=\{ \lambda \in  X_*(T)_\RR : (\lambda,\chi) \geq 0 \}. \]
Let $\rho_T \in X^*(T)$ denote the restriction of $\rho$ to $T$; then a subset 
$W$ of the $T$-weights is called $\rho_T$-semistable if $C_W \subseteq H_{\rho_T}$ 
and otherwise we say $W$ is $\rho_T$-unstable. 

If $W$ is $\rho_T$-unstable, we let $\lambda_W$ be the unique primitive 1-PS in 
$C_W \cap X_*(T)$ for which $\frac{(\lambda,\rho)}{||\lambda||}$ is minimal 
(for the existence and uniqueness of this 1-PS, see \cite{hoskins_quivers} 
Lemma 2.13).
\end{defn}

If $W = \emptyset$, then $C_W = X_*(T)_\RR$ 
and $W$ is $\rho_T$-unstable for all non-trivial $\rho_T$ with  
$\lambda_W = - \rho_T$. 

\begin{prop}\label{prop strat for torus}
Let $T =(\GG_m)^n$ act linearly on an affine variety $V$ 
with respect to a character $\rho : T \ra \GG_m$ and let $|| - ||$ be the 
norm associated to the dot product on $\ZZ^n$; then the Hesselink stratification 
for the torus $T$ is given by
\[ V - V^{\rho-ss} = \bigsqcup_{W \in \cB} S_{\lambda_W}\]
where $\cB = \{ W : W \text{ is } \rho \text{-unstable} \}$ and 
$S_{\lambda_W}= \{ v \in V : \: W = W_v  \}$.
Therefore, the stratification is determined by the $T$-weights and, moreover, 
$ v \in V$ is $\rho$-semistable if and only if its $T$-weight set 
$W_v$ is $\rho$-semistable (that is, $C_{W_v} \subseteq H_{\rho}$).
\end{prop}
\begin{proof}
By construction, $C_{W_v}$ is the cone of (real) 1-PSs $\lambda$ such that 
$\lim_{t \ra 0} \lambda(t) \cdot v$ exists. By the Hilbert--Mumford 
criterion, $v$ is $\rho$-semistable if and only if $(\lambda,\rho) \geq 0$ 
for all $\lambda \in C_{W_v}$; that is, if and only if $C_{W_v} \subseteq H_{\rho}$. 
Therefore, $v$ is $\rho$-semistable if and only if $W_v$ is $\rho$-semistable.

As the conjugation action is trivial, $[\lambda_W]=\lambda_W$ and every $\rho$-unstable 
point has a unique $\rho$-adapted primitive 1-PS. If $v$ is $\rho$-unstable, 
then $\lambda_{W_v}$ is $\rho$-adapted to $v$ by \cite{hoskins_quivers} Lemma 2.13.
\end{proof}

For a reductive group $G$, we fix a maximal torus $T$ of $G$ and positive Weyl 
chamber. Let $\cB$ denote the set of 1-PSs $\lambda_W$ corresponding to 
elements in this positive Weyl chamber where $W$ is a $\rho_T$-unstable set of weights.

\begin{cor}\label{cor index set from torus weights}
Let $G$ be a reductive group acting linearly on an affine variety $V$ with respect to 
a character $\rho : G \ra \GG_m$. Then the Hesselink stratification is given by
\[ V - V^{\rho-ss} = \bigsqcup_{\lambda_W \in \cB} S_{[\lambda_W]}.\]
\end{cor}
\begin{proof}
It is clear that the left hand side is contained in the right hand side. 
Conversely, suppose that $v \in V - V^{\rho-ss}$ and $\lambda$ is 
an primitive 1-PS of $G$ that is $\rho$-adapted to $v$. 
Then there exists $g \in G$ such 
that $\lambda':=g\lambda g^{-1} \in X_*(T)$ and, moreover, $\lambda'$ is 
$\rho$-adapted to $v':= g \cdot v$ for the $G$-action on $V$. As every 1-PSs of $T$ is a 
1-PSs of $G$, it follows that $M^\rho_G(v') \leq M_T^{\rho_T}(v')$. Therefore,
 $\lambda'$ is $\rho_T$-adapted to $v'$ for the $T$-action on $V$ and 
 $\lambda' = \lambda_W \in \cB$, where $W = W_{v'}$. In particular, we have that  
$v \in S_{[\lambda]}= S_{[\lambda_W]}$.
\end{proof}

This gives an algorithm to compute the Hesselink stratification: the index 
set $\cB$ is determined from the $T$-weights and, for each $[\lambda] \in \cB$, we 
compute $S_{[\lambda]}$ from $Z_\lambda$ using Theorem \ref{hess strat thm}.

\begin{rmk}
In general, the Hesselink stratification for $G$ and its maximal torus $T$ are not 
easily comparable. The $\rho$-semistable locus for the $G$-action is contained 
in the $\rho$-semistable locus for the $T$-action, but the $G$-stratification 
does not refine the $T$-stratification. Often there are more $T$-strata, 
since 1-PSs of $T$ may be conjugate in $G$ but not in $T$. 
\end{rmk}

\begin{ex}\label{ex grass}
Let $G:=\GL_r$ act on $V:=\text{Mat}_{r \times n} \cong \AA^{rn}$ by left 
multiplication, linearised with respect to the character $\rho := \det : G \ra \GG_m$. 
It is well known that the GIT quotient is 
\[ \text{Mat}_{r \times n}/\!/_{\det} \GL_r =\text{Gr}(r,n) \]
the Grassmannian of $r$-planes in $\AA^n$ and the GIT semistable locus is 
given by matrices of rank $r$. In this example, we show that the Hesselink 
stratification is given by rank.

We fix the maximal torus $T \cong \GG_m^r$ of diagonal matrices in 
$G$ and use the dot product on $\ZZ^r \cong X_*(T)$ to define a norm 
$|| -||$ on conjugacy classes of 1-PSs. 
The weights of the $T$-action on $V$ are $\chi_1, \dots , \chi_r$ where 
$\chi_i$ denotes the $i^\text{th}$ standard basis vector in $\ZZ^r$ and 
the Weyl group $S_r$ acts transitively on the $T$-weights. The 
restriction of the determinant character to $T$ is given by 
$\rho_T = (1, \dots , 1) \in \ZZ^r \cong X^*(T)$.

As the Weyl group is the full permutation group on the set of $T$-weights, 
it follows that $\lambda_W$ and $\lambda_{W'}$ are conjugate under $S_r$ whenever 
$|W| = |W'|$. Hence, for $k =0, \dots, r$, we let
\[W_k: = \{ \chi_1, \dots , \chi_k \}\]
and note that $W_k$ is $\rho_T$-unstable if and only if $k < r$ with 
corresponding $\rho_T$-adapted 1-PS
\[\begin{smallmatrix}\lambda_k := &(\underbrace{0, \dots, 0,} & 
\underbrace{-1, \dots, -1}).  \\ & k & r-k  \end{smallmatrix}\]
By Corollary \ref{cor index set from torus weights}, 
the index set for the stratification is $\cB= \{ \lambda_k : 0 \leq k \leq r-1 \}$. 
We use Theorem \ref{hess strat thm} to calculate the unstable strata. Since 
\[ V^{\lambda_k}_+ = V^{\lambda_k} = \left\{ v = \left( \begin{array}{ccccc} v_{11}
 & v_{12} & \cdots & \cdots & v_{1n} \\ \vdots & \cdots   & \cdots &\cdots &\vdots \\ 
 v_{k1} & v_{k2} &  \cdots &\cdots & v_{kn} \\ 0 & \cdots  & \cdots & \cdots & 0 \\ 
 \vdots & \cdots &  \cdots & \cdots & \vdots \\ 0 & \cdots &  \cdots & \cdots & 0 
 \end{array} \right) \in \text{Mat}_{r \times n} \right\} \cong \text{Mat}_{k \times n},\]
 we also have that $S_{\lambda_k} = Z_{\lambda_k}$. 
 By definition, $Z_{\lambda_k}$ is the GIT 
 semistable set for $G_{\lambda_k}\cong\GL_k \times \GL_{r-k}$ acting on 
$V^{\lambda_k}$ with respect to the character
\[ \begin{array}{rcc}\rho_{\lambda_k} := 
|| \lambda_k ||^2 \rho - (\lambda_k, \rho) \lambda_k = (r-k) \rho - (k-r) \lambda_k = 
 &(\underbrace{r-k, \dots, r-k,} & \underbrace{0, \dots, 0}).  \\ & k & r-k  \end{array} \]
As positive rescaling of the character does not alter the semistable locus, we can 
without loss of generality assume that $\rho_{\lambda_k}$ is the product of the 
determinant character on $\GL_k$ and the trivial character on $\GL_{r-k}$. 
Therefore, $Z_{\lambda_k}$ is the GIT semistable set for $\GL_k$ acting on 
$\text{Mat}_{k \times n} \cong V^{\lambda_k} $ with respect to $\det : \GL_k \ra \GG_m$; 
that is, $Z_{\lambda_k}$ is the set of matrices in $\text{Mat}_{k \times n}$ whose final 
$r-k$ rows are zero and whose top $k$ row vectors are linearly independent. 
Then, $S_{[\lambda_k]} = GS_{\lambda_k}$ is the locally closed subvariety of rank 
$k$ matrices and the Hesselink stratification is given by rank.
\end{ex}

\section{Stratifications for moduli of quiver representations}\label{sec quiver}

Let $Q=(V,A,h,t)$ be a quiver with vertex set $V$, arrow set $A$ and head 
and tail maps $h,t : A \ra V$ giving the direction of the arrows. 
A $k$-representation of a quiver $Q=(V,A,h,t)$ is a tuple $W = (W_v,\phi_a)$ 
consisting of a $k$-vector space $W_v$ for each vertex $v$ and 
a linear map $\phi_a : W_{t(a)} \ra W_{h(a)}$ for each arrow $a$. 
The dimension vector of $W$ is $ \dim W = (\dim W_v) \in \NN^V$. 

\subsection{Semistable quiver representations}

For quiver representations of a fixed dimension vector $d = (d_v) \in \NN^V$, 
King introduced a notion of 
semistability depending on a stability parameter 
$\theta= (\theta_v) \in \ZZ^V$ such that $\sum_{v\in V} \theta_v d_v = 0$. 

\begin{defn}[King, \cite{king}]\label{theta ss defn}
A representation $W$ of $Q$ of dimension $d$ is $\theta$-semistable if for all 
subrepresentations $W' \subset W$ we have 
$\theta(W'):= \sum_v \theta_v \dim W_v' \geq 0$. 
\end{defn}

We recall King's construction of moduli of $\theta$-semistable 
quiver representations \cite{king}. Let 
\[ \rep_d(Q) := \bigoplus_{a \in A} \text{Hom}(k^{d_{t(a)}},k^{d_{h(a)}}) 
\quad \text{and} \quad  G_d(Q) := \prod_{v \in V} \GL(d_v,k);\]
then $G_d(Q)$ acts on $\rep_d(Q)$ by conjugation. 
Since the diagonal subgroup $\GG_m \cong \Delta \subset G_d(Q)$, given by 
$t \in \GG_m \mapsto (tI_{d_v})_{v \in V} \in G_d(Q)$, acts trivially on $\rep_d(Q)$, 
we often consider the quotient group $\overline{G}_d(Q): = G_d(Q)/\Delta$. 
We note that every representation of $Q$ of dimension $d$ is isomorphic to 
an element of $\rep_d(Q)$ and, moreover, two representations in $\rep_d(Q)$ 
are isomorphic if and only if they lie in the same $\overline{G}_d(Q)$ orbit. 
Hence, the moduli stack $\cR ep_d(Q)$ of 
representations of $Q$ of dimension $d$ is isomorphic to the quotient stack of 
$\overline{G}_d(Q)$ acting on $\rep_d(Q)$:
\[ \cR ep_d(Q) \cong [\rep_d(Q)/\overline{G}_d(Q)].\]
The $G_d(Q)$-action on $\rep_d(Q)$ is linearised by $\rho_{\theta} :G_d(Q) \ra \GG_m$ where 
$\rho_\theta (g_v) := \Pi_v \det(g_v)^{\theta_v}$.

\begin{thm}[King,  \cite{king}]
The moduli space of $\theta$-semistable quiver representations is given by 
\[ M^{\theta-ss}_d(Q):=\rep_d(Q)/\!/_{\rho_\theta} G_d(Q).\]
\end{thm}

In particular, King shows that the notion of $\theta$-semistability for 
quiver representations agrees with the notion of GIT semistability for 
$G_d(Q)$ acting on $\rep_d(Q)$ with respect to $\rho_\theta$.

\begin{rmk} We can also consider moduli spaces of $\theta$-semistable quiver 
representations for quivers with relations. A path in a quiver is a sequence of 
arrows in the quiver $(a_1, \dots, a_n)$ such that $h(a_i) = t(a_{i+1})$ and a 
relation is a linear combination of paths all of which start at some common vertex 
$v_t$ and end at a common vertex $v_h$. Let $\cR$ be a set of relations on $Q$ and let $\rep_d(Q,\cR)$ be the 
closed subvariety of $\rep_d(Q)$ consisting of quiver representations that satisfy 
the relations $\cR$. The moduli space of $\rho$-semistable representations of 
$(Q,\cR)$ is a closed subvariety of the moduli space of $\rho$-semistable 
representations of $Q$:
\[M^{\theta-ss}_d(Q,\cR):=\rep_d(Q,\cR)/\!/_{\rho_\theta}G_d(Q)\subset M^{\theta-ss}_d(Q).\]
\end{rmk}

\subsection{HN filtrations for quivers representations}\label{sec HN filt quiv}

We fix a quiver $Q$, a dimension vector $d \in \NN^V$ and a stability parameter 
$\theta= (\theta_v) \in \ZZ^V$ such that $\sum_v \theta_v d_v = 0$. To 
define a Harder--Narasimhan filtration, we need a notion of semistability for quiver 
representations of all dimension vectors and for this we use a parameter 
$\alpha = (\alpha_v) \in \NN^V$.

\begin{defn}\label{defn HN}
A representation $W$ of $Q$ is $(\theta,\alpha)$-semistable if for all 
$0 \neq W' \subset W$, we have
\[ \frac{\theta(W')}{\alpha(W')} \geq \frac{\theta(W)}{\alpha(W)} \]
where $\alpha(W) := \sum_{v \in V} \alpha_v \dim W_v$. 
A Harder--Narasimhan (HN) filtration of $W$ (with respect to $\theta$ and $\alpha$) 
is a filtration $0 = W^{(0)} \subset W^{(1)} \subset \cdots \subset W^{(s)} =W$ 
by subrepresentations, such that the quotient representations 
$W_i := W^{(i)}/W^{(i-1)}$ are $(\theta,\alpha)$-semistable and
\[ \frac{\theta(W_1)}{\alpha(W_1)} < \frac{\theta(W_2)}{\alpha(W_2)}
      <\cdots < \frac{\theta(W_s)}{\alpha(W_s)} .\]
The Harder--Narasimhan type of $W$ (with respect to $\theta$ and $\alpha$) is 
$\gamma(W) := (\dim W_1, \dots , \dim W_s)$.
\end{defn}

If $W$ is a representation of dimension $d$, then $\theta(W) = 0$ and so $W$ 
is $(\theta,\alpha)$-semistable if and only if it is $\theta$-semistable. 
The proof of the existence and uniqueness of the Harder--Narasimhan 
filtration of quiver representations is standard (for example, see \cite{reineke},  
for the case when $\alpha_v = 1$ for all $v$). 
By the existence and uniqueness of the HN-filtration with respect to $(\theta,\alpha)$, 
we have a decomposition
\[ \rep_d(Q) = \bigsqcup_{\gamma} R_{\gamma} \]
where $R_{\gamma}$ denotes the subset of representations with HN type $\gamma$. 
Reineke \cite{reineke} proves that $R_{\gamma}$ are locally closed subvarieties of 
$\rep_d(Q)$ (when $\alpha_v =1$ for all $v$; the general case is 
analogous). We refer to the above decomposition as the 
HN stratification of 
the representation space (with respect to $\theta$ and $\alpha$). 
As isomorphic quiver representations have the same HN type, each HN stratum 
$R_\gamma$ is invariant under the action of $\overline{G}_d(Q)$. Hence, we 
obtain a HN stratification of the moduli stack of quiver representations:
\[\cR ep_d(Q)=\bigsqcup_{\gamma}\cR ep^{\gamma}_d(Q)\]
where $\cR ep^{\gamma}_d(Q)\cong [R_{\gamma}/\overline{G}_d(Q)]$. We emphasise 
again that this stratification depends on both $\theta$ and $\alpha$. In particular, 
by varying the parameter $\alpha$, we obtain different HN stratifications.

\begin{rmk}\label{rmk HN relns}
For a quiver with relations $(Q,\cR)$, there is also a HN stratification of 
the representation space $\rep_d(Q,\cR)$. If $W$ is a representation of $Q$ 
that satisfies the relations $\cR$, then any subrepresentation of $W$ also 
satisfies these relations. Therefore, the HN strata of $\rep_d(Q,\cR)$ are 
the intersection of the HN strata in $\rep_d(Q)$ with $\rep_d(Q,\cR)$.
\end{rmk}

\subsection{The HN stratification is the Hesselink stratification}

Let $T$ be the maximal torus of $G_d(Q)$ given by the product of the maximal 
tori $T_v \subset \GL(d_v)$ of diagonal matrices. For a rational 1-PS 
$\lambda \in X_*(G)_{\QQ}$, there is a unique  
$C>0$ such that $C\lambda \in X_*(G)$ is primitive.
 
\begin{defn}
For a HN type $\gamma = (d_1, \dots ,d_s)$ of a quiver representation of 
dimension $d$, let $\lambda_\gamma$ be the unique primitive 1-PS associated 
to the rational 1-PS $\lambda'_\gamma = (\lambda'_{\gamma,v})$ of $T$ given by
\[ \lambda'_{\gamma,v}(t) = \text{diag} (t^{r_1}, \dots  , t^{r_1},t^{r_2}, 
\dots  , t^{r_2}, \dots \dots, t^{r_s}, \dots  , t^{r_s} )\]
where the rational weight $r_i:= - \theta(d_i)/\alpha(d_i)$ appears $(d_i)_v$ times. 
\end{defn}

\begin{rmk}
For $\gamma \neq \gamma'$, the conjugacy classes of $\lambda_\gamma$ and 
$\lambda_{\gamma'}$ are distinct.
\end{rmk}

Let $|| - ||_\alpha$ be the norm on 1-PSs of $G_d(Q)$ associated to 
$\alpha \in \NN^V$ (cf. Example \ref{ex norms}). 
For a quiver without relations, the following result is
 \cite{hoskins_quivers} Theorem 5.5. 
 We can deduce the case with relations from this result using
 Theorem \ref{hess strat thm} and Remark \ref{rmk HN relns}.

\begin{thm}\label{quiver HN is Hess}  Let $(Q,\cR)$ be a quiver with relations 
and $d \in \NN^V$ a dimension vector. For stability parameters $\theta \in \ZZ^V$ 
satisfying $\sum_{v \in V} \theta_v d_v = 0$ and $\alpha \in \NN^V$, we let
\[\rep_d(Q,\cR) =   \bigsqcup_{\gamma} R_{\gamma} \quad \quad \text{and} 
\quad \quad  \rep_d(Q,\cR)  = \bigsqcup_{[\lambda]} S_{[\lambda]} \]
denote the HN stratification with respect to $(\theta,\alpha)$ and the 
Hesselink stratification with respect to $\rho_\theta$ and $||-||_\alpha$ 
respectively. Then $R_\gamma = S_{[\lambda_\gamma]}$ and the stratifications 
coincide.
\end{thm}

The moduli stack admits a finite stratification by HN types or, equivalently, 
Hesselink strata
\[ \cR ep_d(Q,\cR) = \bigsqcup_{\tau} \cR ep_d^\gamma(Q,\cR) = 
\bigsqcup_{[\lambda]} [S_{[\lambda]}/\overline{G}_d(Q)].\]

\section{Projective GIT and Hesselink stratifications}

Let $G$ be a reductive group acting on a projective scheme $X$ with respect to an 
ample $G$-linearisation $L$. The projective GIT quotient 
\[X^{ss}(L) \ra X/\!/_L G:= \proj \bigoplus_{n \geq 0} H^0(X, {L}_\rho^{\otimes n})^G\]
is a good quotient of the open subscheme $X^{ss}(L) $ of $ X$ of $L$-semistable points. 
For a 1-PS $\lambda$ of $G$ and $x \in X$, we 
let $\mu^L(x,\lambda) $ denote the weight of the $\GG_m$ action on the fibre of 
$L$ over the fixed point $\lim_{t \ra 0} \lambda(t) \cdot x$. 

\begin{thm}[Hilbert--Mumford criterion \cite{mumford}]
A point $x$ is $L$-semistable if and only we have $\mu^L(x,\lambda) \geq 0$ for 
all 1-PSs $\lambda : \GG_m \ra G$.
\end{thm}

Fix a norm $|| - ||$ on the set $\overline{X}_*(G)$ of conjugacy classes of 1-PSs 
such that $|| - ||^2$ is $\ZZ$-valued. 

\begin{defn}[Kempf, \cite{kempf}]
For $x \in X$, let
\[ M^L(x) := \inf\left\{\frac{\mu^L(x,\lambda)}{|| \lambda ||} : \:\text{1-PSs}
     \: \lambda \text{ of } G  \right\}.\]
A 1-PS $\lambda$ is $L$-adapted to an $L$-unstable point $x$ if
\[ M^L(x) = \frac{\mu^L(x,\lambda)}{|| \lambda ||} . \]
Let $\wedge^L(x) $ denote the set of primitive 1-PSs which are $L$-adapted to $x$. 
\end{defn}

\begin{defn}(Hesselink, \cite{hesselink})
For a conjugacy class $[\lambda]$ of 1-PSs of $G$ and $ d \in \QQ_{> 0}$, we let
\[S_{[\lambda],d}:= \{x: M^L(x)=-d \: \text{ and }  \wedge^\rho(x) \cap [\lambda] \neq \emptyset\}.\]
\end{defn}

Let $p_\lambda : X \ra X^\lambda$ denote the retraction onto the $\lambda$-fixed locus 
given by $x \mapsto \lim_{t \ra 0} \lambda(t) \cdot x$ and let 
$X_d^\lambda \subset X^\lambda$ 
be the union of components on which $\mu^L(-,\lambda)=-d||\lambda||$. 
Let $G_\lambda$ be the subgroup of elements of $G$ that commute with $\lambda$ and let 
$P(\lambda) $ be the parabolic subgroup of elements $g \in G$ such that  
$\lim_{t \ra 0} \lambda(t)g \lambda(t)^{-1}$ exists in $G$. 

\begin{thm}[cf. \cite{hesselink,kirwan,ness}]\label{hess thm proj}
Let $G$ be a reductive group acting on a projective scheme $X$ with respect to an 
ample $G$-linearisation $L$ and let $|| - ||$ be a norm on conjugacy classes of 1-PSs of $G$. 
Then there is a Hesselink stratification
 \[ X  -X^{ss}(L)= \bigsqcup_{([\lambda],d)} S_{[\lambda],d} \]
into finitely many disjoint $G$-invariant locally closed subschemes of $X$. Moreover:
\begin{enumerate}
\item $S_{[\lambda],d} = GY^{\lambda}_d \cong G \times_{P(\lambda)} Y^\lambda_d$ where 
$Y^{\lambda}_d =\{ x \in X: M^L(x) = -d \: \text{ and } \: \lambda  \in \wedge^\rho(x)\}$;
\item $Y^\lambda_d = p_\lambda^{-1}(Z_d^\lambda)$ where 
$Z_d^\lambda= \{ x \in X^\lambda: M^L(x) = -d \: \text{ and } \: \lambda  \in \wedge^\rho(x)\}$;
\item $Z_d^\lambda$ is the GIT-semistable locus for the reductive group $G_\lambda$ acting on 
$X_d^\lambda$ with respect to a modified linearisation $L_\beta$;
\item $\partial S_{[\lambda],d} \cap S_{{[\lambda'],d'}} \neq \emptyset$ only if $d > d'$.
\end{enumerate}
\end{thm}

Let $S_0:= X^{ss}(L)$; then we obtain a Hesselink stratification of $X$ 
\[ X = \bigsqcup_{\beta \in \cB} S_\beta\]
where the lowest stratum is $S_0 $ and the higher strata are indexed by pairs $\beta = ([\lambda],d)$ as above. We note that if $\beta = ([\lambda],d)$ indexes a non-empty 
stratum, then $d ||\lambda|| \in \ZZ$ as $\mu(-,\lambda)$ is $\ZZ$-valued.

\begin{rmk}\label{new indices for H strat}
An unstable Hesselink index $\beta =([\lambda],d)$ determines a rational 1-PS
\[\lambda_\beta := \frac{d}{||\lambda||} \lambda\]
(where $d/||\lambda||$ is rational, as $||\lambda||^2$ and $d||\lambda||$ are 
both integral). 
In fact, we can recover $([\lambda],d)$ from $\lambda_\beta$, as 
$d = - ||\lambda_\beta||$ and $\lambda $ is the unique primitive 1-PS lying on 
the ray spanned by $\lambda_\beta$. 
\end{rmk}

\begin{rmk}\label{rmk on hess strat}
The additional index $d$ is redundant in the affine case, as it can be 
determined from the character $\rho$ and 1-PS $\lambda$ (more precisely, we have $d = - (\rho,\lambda)/||\lambda||$). Moreover, projective Hesselink stratifications share the following properties with affine Hesselink 
stratifications. 
\begin{enumerate}
\item If $X \subset \PP^n$ and $G$ acts via a representation 
$\rho : G \ra \GL_{n+1}$, then the unstable indices can be 
determined from the weights of a maximal torus of $G$; 
for details, see \cite{kirwan}.
\item For a closed $G$-invariant subscheme $X' $ of $X$, the Hesselink 
stratification of $X'$ is the intersection of the Hesselink stratification of $X$ with $X'$.
\item The Hesselink strata may not be connected; 
however, if $Z_{d,i}^\lambda$ denote the connected components of $Z_d^\lambda$, then $S_{[\lambda],d,i} := G p_\lambda^{-1}(Z_{d,i}^\lambda)$ are the connected components of $S_{[\lambda],d}.$
\end{enumerate}
\end{rmk}

\section{Stratifications for moduli of sheaves}\label{sec sheaves}

Let $(X, \cO_X(1))$ be a projective scheme of finite type over $k$ with a very ample line bundle. 

\subsection{Preliminaries}
For a sheaf $\cE$ over $X$, we let
$\cE(n) := \cE \otimes \cO_X(n)$ and let
\[ P(\cE,n) = \chi(\cE(n)) = \sum_{i=0}^n (-1)^i \dim H^i(\cE(n))\]
be the Hilbert polynomial of $\cE$ with respect to $\cO_X(1)$. 
We recall that the degree of $P(\cE)$ 
is equal to the dimension of $\cE$ (i.e., the dimension of 
the support of $\cE$) and a sheaf $\cE$ is pure if all its non-zero subsheaves 
have the same dimension as $\cE$. For a non-zero sheaf $\cE$, 
the leading coefficient $r(\cE)$ 
of $P(\cE)$ is positive and we define the reduced Hilbert polynomial of $\cE$ to be 
$P^{\red}(\cE):=P(\cE)/\text{r}(\cE)$.

\begin{defn}[Castelnuovo-Mumford \cite{mumfordcurves}]
A sheaf $\cE$ over $X$ is $n$-regular if 
\[ H^{i}(\cE(n-i)) = 0  \quad \text{ for all } \: i > 0. \]
\end{defn}

This is an open condition and, by Serre's Vanishing Theorem, any sheaf is $n$-regular 
for $n > \!> 0$. Furthermore, any bounded family of sheaves is $n$-regular for $n > \!> 0$. 

\begin{lemma}[cf. \cite{mumfordcurves}]
For a $n$-regular sheaf $\cE$ over $X$, we have the following results.
\begin{enumerate}
\renewcommand{\labelenumi}{\roman{enumi})}
\item $\cE$ is $m$-regular for all $m \geq n$.
\item $\cE(n)$ is globally generated with vanishing higher cohomology, i.e., the evaluation map 
$H^0(\cE(n)) \otimes \cO_X(-n) \ra \cE$ is surjective and $H^i(\cE(n)) = 0$ for $i > 0$.
\item The natural multiplication maps 
$H^0(\cE(n)) \otimes H^0(\cO_X(m-n)) \ra H^0(\cE(m))$ are surjective for 
all $m \geq n$.
\end{enumerate}
If we have an exact 
sequence of sheaves over $X$ 
\[ 0 \ra \cF' \ra \cF \ra \cF'' \ra 0,\]
such that $\cF'$ and $\cF''$ are both $n$-regular, then $\cF$ is also $n$-regular.
\end{lemma}

\subsection{Construction of the moduli space of semistable sheaves}\label{sec simp constr}

In this section, we outline Simpson's construction \cite{simpson} of the moduli space of semistable 
sheaves on $(X,\cO_X(1))$ with Hilbert polynomial $P$. 
We define an ordering $ \leq $ on rational polynomials in one variable by 
$P \leq Q$ if $P(x) \leq Q(x)$ for all sufficiently large $x$. For polynomials of a 
fixed degree with positive leading coefficient, this is equivalent to the 
lexicographic ordering on the coefficients.

\begin{defn}\label{defn ss}
 A pure sheaf $\cF$ on $X$ is semistable (in the sense of Gieseker) if 
 for all non-zero subsheaves $\cE \subset \cF$, we have 
 $P^{\red}(\cE) \leq P^{\red}(\cF).$
\end{defn}

By the Simpson--Le Potier bounds (cf. \cite{simpson} Theorem 1.1), 
the set of semistable sheaves with Hilbert polynomial $P$ is bounded; therefore, 
for $n >\!> 0$, all semistable sheaves with Hilbert polynomial $P$ are $n$-regular. 

Let $V_n := k^{P(n)}$ be the trivial $P(n)$-dimensional vector space and 
let $\quot(V_n \otimes \cO_X(-n), P)$ denote the 
Quot scheme parametrising quotients sheaves of $V_n \otimes  \cO_X(-n)$ 
with Hilbert polynomial $P$. Let $Q_n$ denote the open subscheme of this 
Quot scheme consisting of quotients $q : V_n \otimes \cO_X(-n) \onto \cE$ such that 
$H^0(q(n))$ is an isomorphism, let $Q_n^{\pur} \subset Q_n$ denote the open subscheme 
consisting of quotient sheaves that are pure and let $R_n$ denote the closure of 
$Q_n^{\pur}$ in the quot scheme.

Let $\cE$ be a $n$-regular sheaf; then $\cE(n)$ is globally generated 
and $\dim H^0(\cE(n)) = P(\cE,n)$. 
In particular, every $n$-regular sheaf on $X$ with Hilbert polynomial 
$P$ can be represented as a quotient sheaf in $Q_n$ by choosing an isomorphism 
$H^0(\cE(n)) \cong V_n$ and using the surjective evaluation map 
$H^0(\cE(n) ) \otimes \cO_X(-n) \onto \cE$. 

The group $G_n:=\GL(V_n)$ acts on $V_n$ and $Q_n$ such that the $G_n$-orbits in 
$Q_n$ are in bijection with isomorphism classes of sheaves in $Q_n$. 
As the diagonal $\GG_m$ acts trivially, we often consider the 
action of $\SL(V_n)$. The action is linearised using Grothendieck's 
embedding of the Quot scheme into a Grassmannian: the corresponding line
bundle is given by
\[ L_{n,m} := \det(\pi_{*}(\cU_n \otimes \pi_X^*\cO_X(m))) \]
for $m >\!> n$, where $\cU_n$ denotes the universal quotient sheaf on the Quot scheme. 

\begin{thm}[Simpson \cite{simpson}, Theorem 1.21]\label{simthm}
For $m>\!>n>\!>0$, the moduli space of semistable sheaves on $X$ 
with Hilbert polynomial $P$ is the GIT quotient of $R_n$; that is,
\[ M^\text{ss}(X,P)=R_n /\!/_{L_{n,m}} \SL(V_n).\]
\end{thm}

In his proof, Simpson shows, for $m>\!>n>\!>0$ (depending on  $X$ and $P$), 
that an element $q : V_n \otimes \cO_X(-n) \onto \cE$ in $R_n$ is GIT semistable 
if and only if $q$ belongs to the open subscheme $Q_n^{ss}$ of $Q_n$ consisting of quotient sheaves 
$q: V_n \otimes \cO_X(-n) \onto \cE$ such that $\cE$ is a semistable sheaf.

\subsection{The Hesselink stratification of the Quot scheme}\label{sec Hess quot}

Associated to the action of $\SL(V_n)$ on $\quot(V_n \otimes \cO_X(-n),P)$ with 
respect to 
$L_{n,m}$ and the norm $||-||$ coming from the dot product on the diagonal torus $T$ 
(cf. Example \ref{ex norms}), there is a Hesselink stratification
\[ \quot(V_n \otimes \cO_X(-n),P) = \bigsqcup_{\beta \in \cB} S_\beta \]
In this section, for fixed $n$, we describe this stratification of  
$\quot := \quot(V_n \otimes \cO_X(-n),P)$; this completes the partial description 
of the Hesselink stratification on $R_n$ given in \cite{hoskinskirwan}. 

By Remark \ref{new indices for H strat}, the unstable indices $\beta$ 
can equivalently be viewed as conjugacy classes of rational 1-PSs 
$[\lambda_\beta]$ of $\SL(V_n)$. We take a representative $\lambda_\beta \in X_*(T)$ 
and, by using the Weyl group action, assume that the weights are decreasing; that is,
\[ \lambda_\beta(t) = \diag (t^{r_1}, \dots , t^{r_1},t^{r_2}, 
\dots , t^{r_2} , \dots , t^{r_s} ,\dots ,t^{r_s}) \]
where $r_i$ are strictly decreasing rational numbers that occur with multiplicities $l_i$. Then 
$\beta$ is equivalent to the decreasing sequence of rational weights 
$r(\beta): = (r_1, \dots , r_s)$ and multiplicities $l(\beta):=(l_1, \dots ,l_s)$. 
Since $\lambda_\beta $ is a rational 1-PSs of $\SL(V_n) $, we note that 
\[ \sum_{i=1}^s l_i = \dim V_n =P(n)\quad\quad \text{and} \quad\quad \sum_{i=1}^s r_il_i = 0.\]

We first describe the fixed locus for $\lambda_\beta$. The 1-PS $\lambda_\beta$ 
determines a filtration of $V:=V_n$
\[ 0 = V^{(0)} \subset V^{(1)} \subset \cdots \subset V^{(s)} =V, 
\quad \text{where} \:\: V^{(i)}:=k^{l_1 + \cdots +l_i} \]
with $V^i := V^{(i)}/V^{(i-1)}$ of dimension $l_i$.
For a quotient $q: V \otimes \cO_X(-n) \onto \cE$ in $\quot$, we let 
$\cE^{(i)}:=q(V^{(i)}\otimes \cO_X(-n))$; then we have exact sequences of quotient sheaves
\[\xymatrix{ 0 \ar[r] & V^{(i-1)} \otimes \cO_X(-n) \ar@{->>}[d]^{q^{(i-1)}} \ar[r] & 
V^{(i)} \otimes \cO_X(-n)\ar@{->>}[d]^{q^{(i)}}  \ar[r] & V^i \otimes \cO_X(-n) 
\ar@{->>}[d]^{q^{i}} \ar[r] & 0  \\ 0 \ar[r] & \cE^{(i-1)}  \ar[r] & \cE^{(i)}  
\ar[r] & \cE^i:= \cE^{(i)}/\cE^{(i-1)} \ar[r] & 0.}\]
We fix an isomorphism $V \cong \oplus_{i=1}^s V^i$; then via this 
isomorphism we have  
$\lim_{t \ra 0} \lambda_\beta(t) \cdot q = \oplus_{i=1}^s q_i $ 
by \cite{huybrechts} Lemma 4.4.3. 

\begin{lemma}\label{lemma on fixed locus}
For $\beta=([\lambda],d)$ with $r(\beta)=(r_1, \dots, r_s)$ and $l(\beta)=(l_1, \dots, l_s)$ as above, we have
\[ \quot^{\lambda} \: \: = \bigsqcup_{(P_1, \dots ,P_s) \: : \: 
 \sum_{i=1}^s P_i = P } \: \: \:    \prod_{i=1}^s \quot(V^i \otimes \cO_X(-n),P_i).\]
and
\[\quot^{\lambda}_d \:\:\cong
\bigsqcup_{\begin{smallmatrix}(P_1, \dots ,P_s)\: : \: \sum_{i=1}^s P_i = P, 
\\ \sum_{i=1}^s r_i P_i(m) + r_i^2 l_i = 0 \end{smallmatrix}}
 \: \: \: \prod_{i=1}^s \quot(V^i \otimes \cO_X(-n),P_i).\]
\end{lemma}
\begin{proof} 
The description of the $\lambda$-fixed locus follows from the discussion 
above. By definition,  
$\quot^{\lambda}_d$ is the union of fixed locus components on which 
$\mu(-,\lambda)=-d ||\lambda||$ or, equivalently, 
\[ \mu(-,\lambda_\beta) = - ||\lambda_\beta||^2 = - \sum_{i=1}^s r_i^2 l_i,\]
since $\lambda_\beta = (d/|| \lambda||) \lambda$. For a $\lambda$-fixed quotient sheaf 
$q : V \otimes \cO_X(-n) \onto \oplus_{i=1}^s \cE^i$, we have that
$ \mu(q,\lambda_\beta) = \sum_{i=1}^s r_i P(\cE^i,m)$
by \cite{huybrechts} Lemma 4.4.4 and so this completes the proof.
\end{proof}

By Theorem \ref{hess thm proj}, for $\beta = ([\lambda],d)$ as above, we 
can construct the stratum $S_\beta$ from the limit set $Z_d^\lambda$ which is the 
GIT semistable set for the subgroup
\[ G_\lambda \cong \left\{ (g_1, \dots , g_s) \in \prod_{i=1}^s \GL({V^i}) 
: \prod_{i=1}^s \det g_i = 1 \right\}\]
acting on $\quot^{\lambda}_d$ with respect to a modified linearisation $L_\beta$, 
obtained by twisting the original linearisation 
$L = L_{n,m}$ by the (rational) character $\chi_\beta : G_\lambda \ra \GG_m$ 
where 
\[ \chi_\beta(g_1, \dots , g_s) =\prod_{i=1}^s (\det g_i)^{r_i}.\]

\begin{prop}\label{Zdlambda}
Let $\beta=([\lambda],d)$ with $r(\beta)=(r_1, \dots, r_s)$ and 
$l(\beta)=(l_1, \dots, l_s)$ as above; then
\[ Z_d^\lambda\: \: \cong \bigsqcup_{(P_1, \dots ,P_s) } 
\prod_{i=1}^s \quot(V^i \otimes \cO_X(-n),P_i)^{\SL(V^i)-ss}(L_{n,m})\]
where this is a union over tuples $(P_1, \dots ,P_s)$ such that $\sum_{i=1}^s P_i = P$ and 
\begin{equation}\label{num cond} r_i = \frac{P(m)}{P(n)} - \frac{P_i(m)}{l_i} 
\quad \quad \text{for } \: i = 1, \dots , s.
\end{equation}
\end{prop}
\begin{proof}
The centre 
$ Z(G_\lambda)  \cong \{ (t_1, \dots , t_s) \in (\GG_m)^s : \Pi_{i=1}^s t_i^{l_i} = 1\}$ 
acts trivially on $\quot^{\lambda} $ and acts on the fibre of $L$ over $q \in \quot^\lambda$, 
lying in a component indexed by $(P_1, \dots , P_s)$, by a character 
\[ \chi_q : Z(G_\lambda) \ra \GG_m \quad \text{ given by } \quad 
\chi_{q}(t_1, \dots , t_s)= \prod_{i=1}^s t_i^{P_i(m)}. \]
Therefore, $Z(G_\lambda)$ acts on the fibre of 
$L_\beta$ over $q$ by $\chi_q \chi_\beta(t_1, \dots , t_s) = \prod_{i=1}^s t_i^{P_i(m) +r_i l_i}$. 
By the Hilbert--Mumford criterion, $q$ is unstable for the action of $G_\lambda$ 
with respect to $L_\beta$ unless $\chi_q \chi_\beta$ is trivial; i.e., 
there is a constant $C$ such that $ C l_i = P_i(m) + r_i l_i$, 
for $i=1, \dots , s$. In this case
\[ C P(n) =\sum_{i=1}^s C l_i = \sum_{i=1}^s P_i(m) + r_i l_i = P(m)\]
and we see that the conditions given at (\ref{num cond}) are necessary 
for $q$ to belong to $Z_d^\lambda$.

We suppose that $(P_1, \dots ,P_s)$ is a tuple satisfying $\sum_{i=1}^s P_i = P$ and 
the conditions given at (\ref{num cond}). 
Since $G_\lambda$ is, modulo a finite group, 
$\Pi_{i=1}^s \SL(V^i) \times Z(G_\lambda)$ and $Z(G_\lambda)$ acts on both 
\[ \prod_{i=1}^s \quot(V^i \otimes \cO_X(-n),P_i)\]
and $L_\beta$ trivially, the semistable locus for $G_\lambda$ is equal to the 
semistable locus for $\Pi_{i=1}^s \SL(V^i)$. As $\Pi_{i=1}^s \SL(V^i)$-linearisations, we have that $L_\beta=L$. 
It then follows that 
\[ \left(\prod_{i=1}^s \quot(V^i \otimes \cO_X(-n),P_i) \right)^{\Pi_{i=1}^s \SL(V^i)-ss}(L)
= \prod_{i=1}^s \quot(V^i \otimes \cO_X(-n),P_i)^{\SL(V^i)-ss}(L)\]
(for example, see \cite{hoskinskirwan} Lemma 6.6) which completes the proof of the proposition.
\end{proof}

The following corollary is an immediate consequence of Theorem \ref{hess thm proj}.

\begin{cor}\label{cor Hess quot}
For $\beta$ as above, the scheme $S_\beta$ parametrises quotients 
$q : V \otimes \cO_X(-n) \onto \cE$ with a filtration 
$0 = W^{(0)} \subset W^{(1)} \subset \cdots \subset W^{(s)} =V$ 
such that, for $ i =1 , \dots , s$, we have
\begin{enumerate}
\item $\dim W^i = l_i$, where $W^{i}:=W^{(i)}/W^{(i-1)}$,
\item $P(n){P(\cE^i,m)}=l_i({P(m)} - r_iP(n)) $, and
\item the quotient sheaf $q^i : W^i \otimes \cO_X(-n) \onto \cE^i$ is 
$\SL(W^i)$-semistable with respect to $L_{n,m}$,
\end{enumerate} 
where $\cE^{(i)}:=q(W^{(i)} \otimes \cO_X(-n))$ and $\cE^i:=\cE^{(i)}/\cE^{(i-1)}$. 
Furthermore, $S_\beta = \SL(V)Y_d^\lambda$ where 
$Y^d_\lambda$ parametrises quotients $q$ with a filtration $W^\bullet$ as above such that $W^\bullet = V^\bullet$.
\end{cor}

\begin{rmk}\label{refine Hess}
We see from the above description of the Hesselink strata that the strata are not always connected (cf. 
Remark \ref{rmk on hess strat}) and it is natural to further decompose the limit 
sets $Z_d^\lambda$ as
\[ Z_d^\lambda = \bigsqcup_{\underline{P} \in \cC_\beta} Z_{\underline{P}} 
\quad \quad \text{where } \:
Z_{\underline{P}} \cong \prod_{i=1}^s \quot(V^i \otimes \cO_X(-n),P_i)^{\SL(V^i)-ss}(L_{n,m}).\]
Here $\cC_\beta$ is the set of tuples $\underline{P}=(P_1, \dots ,P_s)$ 
such that $\sum_{i=1}^s P_i = P$, the conditions (\ref{num cond}) hold and 
$Z_{\underline{P}}$ is non-empty. For $\underline{P}=(P_1, \dots ,P_s) \in \cC_\beta$, we note that
\begin{equation}\label{cond on tuple} 
\frac{P_1(n)}{P_1(m)} > \cdots \cdots > \frac{P_s(n)}{P_s(m)}. 
\end{equation}
Then $S_\beta$ is a disjoint union of the schemes 
$S_{\underline{P}}:= \SL(V) p_\lambda^{-1}(Z_{\underline{P}})$ over 
$\underline{P} \in \cC_\beta$ and we obtain a refinement of the 
Hesselink stratification:
\[ \quot(V_n \otimes \cO_X(-n),P)= \bigsqcup_{\underline{P}} S_{\underline{P}}.\] 
\end{rmk}

\subsection{HN filtrations for coherent sheaves}

In this section, we describe a canonical destabilising filtration for each coherent 
sheaf, known as its Harder--Narasimhan (HN) filtration \cite{hn}. 

\begin{defn}\label{def HN usual} 
 A pure HN filtration of a sheaf $\cF$ is a filtration by subsheaves
\[ 0 = \cF^{(0)} \subset  \cF^{(1)} \subset \cdots \subset \cF^{(s)} = \cF \]
such that $\cF_i := \cF^{(i)}/\cF^{(i-1)}$ are semistable and 
$ P^{\red}(\cF_1) > P^{\red}(\cF_2) > \cdots > P^{\red}(\cF_s)$.
\end{defn}

By \cite{huybrechts} Theorem 1.3.4, every pure sheaf has a unique pure HN filtration. 
To define HN filtrations for non-pure sheaves, we need an alternative 
definition of semistability that does not use reduced Hilbert polynomials, 
as every non-pure sheaf $\cE$ has a non-zero subsheaf whose Hilbert polynomial 
has degree strictly less than that of $\cE$. For this, we use an extended 
notion of semistability due to Rudakov \cite{rudakov}. 
We define a partial ordering $\preccurlyeq$ on $\QQ[x]$ by
\[ P \preccurlyeq Q  \iff \frac{P(n)}{P(m)} \leq \frac{Q(n)}{Q(m)} 
\quad \quad \text{for \:} \: m >\! > n >\! > 0\]
and similarly define a strict partial order $\prec$ by replacing $\leq$ with $<$. 
This ordering allows us to compare polynomials with positive leading coefficient of 
different degrees in a way that polynomials of lower degree are bigger with 
respect to this ordering.  

\begin{rmk} Rudakov formulated this preordering using the coefficients of the polynomials: 
for rational polynomials $P(x) = p_d x^d + \cdots + p_0$ and 
$Q(x) = q_e x^e + \cdots +q_0$, let
\[ \Lambda(P,Q):=(\lambda_{f,f-1}, \dots , \lambda_{f,0}, \lambda_{f-1,f-2}, \dots , 
\lambda_{f-1,0}, \dots , \lambda_{1,0}) \quad \text{where} \quad \lambda_{i,j}:=p_iq_j - q_ip_j\]
and $f := \max(d,e)$. 
We write $\Lambda(P,Q)>0$ if the first non-zero $\lambda_{i,j}$ appearing in 
$\Lambda(P,Q)$ is positive. Then $ P \preccurlyeq Q$ is equivalent to $ \Lambda(P,Q) \geq 0 $. 
\end{rmk}

\begin{defn}\label{def rud ss}
 A sheaf $\cF$ is semistable if $P(\cE) \preccurlyeq P(\cF)$ for all 
 non-zero subsheaves $\cE \subset \cF$.
\end{defn}

This definition of semistability implies purity, as polynomials 
of smaller degree are bigger with respect to $\preccurlyeq$. Moreover, for Hilbert 
polynomials $P(\cE)$ and $P(\cF)$ of the same degree, we have $P(\cE) \preccurlyeq P(\cF)$ 
if and only if $P^{\red}(\cE) \leq P^{\red}(\cF)$. Thus, a sheaf is semistable in the 
sense of Definition \ref{def rud ss} if and only if it is semistable in the sense of 
Definition \ref{defn ss}. 

\begin{defn}\label{def HN gen} 
 A HN filtration of a sheaf $\cE$ is a filtration by subsheaves
\[ 0 = \cE^{(0)} \subset  \cE^{(1)} \subset \cdots \subset \cE^{(s)} = \cE \]
such that $\cE_i := \cE^{(i)}/\cE^{(i-1)}$ are semistable and 
$ {P(\cE_1)} \succ {P(\cE_2)} \succ \cdots \succ {P(\cE_s)}$ . 
\end{defn}

Using this definition of HN filtration, every coherent sheaf has a 
unique HN filtration (cf. \cite{rudakov}, Corollary 28). 
We can construct the HN filtration 
of a sheaf from its torsion filtration and the pure HN filtrations 
of the subquotients in the torsion filtration as follows.

\begin{prop}\label{prop HN gen}
Let $0 \subset T^{(0)}(\cE) \subset  \cdots \subset T^{(d)}(\cE) = \cE$ be the torsion 
filtration of $\cE$ and let
\begin{equation}\label{HN tor succ}
  0 = \cF^{(0)}_i \subset  \cF^{(1)}_i \subset \cdots \subset \cF^{(s_i)}_i = 
  T_i(\cE) :=T^{(i)}(\cE)/T^{(i-1)}(\cE) 
\end{equation}
be the pure HN filtrations of the subquotients in the torsion filtration. 
 Then $\cE$ has HN filtration
\[0 = \cE_0^{(0)} \subset \cE_0^{(1)} \subset \cE_1^{(1)} \subset \cdots \subset \cE_1^{(s_1)} \subset \cdots \subset \cE_d^{(1)} \subset \cdots \subset \cE_d^{(s_d)} =\cE \]
where $\cE^{(j)}_i$ is the preimage of $\cF^{(j)}_i$ under the quotient map $T^{(i)}(\cE) \ra T_i(\cE)$.
\end{prop}
\begin{proof}
 As the HN filtration is unique, it suffices to check that the subquotients are 
 semistable with decreasing Hilbert polynomials for $\prec$. 
 First, we note that 
 $\cE^{(0)}_i = T^{(i-1)}(\cE)= \cE^{(s_{i-1})}_{i-1}$ and 
\[ \cE_i^j:=\cE_i^{(j)}/\cE_i^{(j-1)} \cong \cF_i^{(j)}/\cF_i^{(j-1)}=:\cF_i^{j}.\]
Since (\ref{HN tor succ}) is the pure HN filtration of $T_i(\cE)$, we have inequalities  
$P^{\red}(\cF_i^1)> \cdots > P^{\red}(\cF_i^{s_i})$ and the subquotients $\cF_i^{j}$ 
are semistable. Moreover, it follows that 
\[  P(\cE_0^1) \succ P(\cE_1^1) \succ \cdots \succ P(\cE_1^{s_1}) \succ \cdots  \succ P(\cE_d^1) \succ \cdots \succ P(\cE_d^{s_d}), \]
as $\deg P(\cF_i^j) =i$ and polynomials of lower degree are bigger with respect to this ordering.
\end{proof}

\begin{defn}
Let $\cE$ be a sheaf with HN filtration 
$ 0 = \cE^{(0)} \subset  \cE^{(1)} \subset \cdots \subset \cE^{(s)} = \cE $; 
then the HN type of $\cE$ is the tuple $\tau(\cE):=(P(\cE_{1}),\dots ,P(\cE_{s}))$ 
where $\cE_i := \cE^{(i)}/\cE^{(i-1)}$. 
We say $\tau = (P_1,\dots, P_s)$ is a pure HN type if all polynomials $P_i$ have the same degree.
\end{defn}

\subsection{The HN stratification on the stack of coherent sheaves}

Let $\cC oh_{P}(X)$ denote the stack of sheaves on $X$ with Hilbert polynomial $P$; this is 
an Artin stack such that
 \begin{equation}\label{coh as union} \cC oh_{P}(X) \cong \bigcup_{n} [Q_n^o/G_n]
 \end{equation}
where $G_n = \GL(V_n)$ and $Q_n^o$ is the open subscheme of $\quot(V_n \otimes \cO_X(-n),P)$ 
 consisting of quotient sheaves $q : V_n \otimes \cO_X(-n) \onto \cE$ such that $H^0(q(n))$ 
 is an isomorphism and $H^i(\cE(n)) = 0 $ for $i > 0$ (cf. \cite{LMB} 
 Theorem 4.6.2.1). 
 
Let
\[ \cH_{P}:= \left\{ \tau= (P_1,\dots, P_s) : \: P=\sum_{i=1}^s P_i   \: \text{and}
 \:  {P_1} \succ {P_2} \succ \cdots \succ {P_s} \right\}\]
 be the set of HN types of sheaves with Hilbert polynomial $P$. 
For a HN type $\tau= (P_1,\dots, P_s)$,  we define the $(n,m)$th Shatz polytope $\Gamma(\tau,n,m)$ 
to be the union of the
line segments joining $x_k: = (\sum_{j=1}^k P_j(m), \sum_{j=1}^k P_j(n))$ to $x_{k+1}$ 
for $ k =0, \dots , s-1$.  We define a partial order $\leq$ on the set of HN types $\cH_P$ by $\tau \leq \tau'$ if $\Gamma(\tau,n,m) $ 
lies above $\Gamma(\tau',n,m) $ for $m >\! > n >\!> 0$. 

\begin{thm}[Shatz \cite{shatz}; see also \cite{nitsureshn} Theorem 5]
Let $\cF$ be a family of sheaves on $X$ with Hilbert polynomial $P$ 
parametrised by a scheme $S$; then the HN type function $ S \ra \cH_P$ 
given by $s \mapsto \tau(\cF_s)$ is upper semi-continuous. Therefore,
\begin{enumerate}
\item $S_{> \tau}= \{ s \in S : \tau(\cF_s) > \tau \}$ is closed in $S$,
\item $S_\tau= \{ s \in S : \tau(\cF_s)= \tau \}$ is 
locally closed in $S$,
\end{enumerate}
and, moreover, there is a finite Shatz stratification of $S$ into disjoint 
subschemes $S_\tau$ such that 
\[\overline{S_\tau} \subset \bigsqcup_{\tau' \geq \tau} S_{\tau'}.\]
\end{thm}

The universal quotient sheaf $\cU_n$ over $X \times Q_n^o$ is a family of 
sheaves on $X$ with Hilbert polynomial $P$ parametrised by $Q_n^o$; therefore, we have 
an associated Shatz stratification
\begin{equation}\label{n shatz strat of quot} 
Q_n^o = \bigsqcup_\tau Q_{n,\tau}.
\end{equation}
As the Shatz strata $Q_{n,\tau}$ are 
$G_n$-invariant, this stratification descends to the stack quotient
\[ [Q_n^o/G_n] = \bigsqcup_\tau [Q_{n,\tau}/G_n].\]
From the description (\ref{coh as union}) of $\cC oh_P(X)$, 
we have the following corollary.

\begin{cor}\label{Shatz strat}
There is a HN stratification on the stack of coherent sheaves 
 \[\cC oh_{P}(X) = \bigsqcup_{\tau \in \cH_P} \cC oh_{P}^\tau(X) \]
into disjoint locally closed substacks $\cC oh_{P}^\tau(X)$ such that 
$ \overline{\cC oh_{P}^\tau(X)} \subset \bigsqcup_{\tau' \geq \tau} \cC oh_{P}^{\tau'}(X).$
\end{cor}

\begin{rmk}\label{mspace}
If $\tau = (P)$ is the trivial HN type, then $\cC oh^\tau_{X,P}=\cC oh^{ss}_{X,P}$ 
and, for $n >\!> 0$,  
 \[\cC oh^{ss}_{X,P} \cong [Q^{ss}_n / G_n]\]
where $Q^{ss}_n =Q_{n,(P)}$ is an open subscheme of 
$\quot(V_n \otimes \cO_X(-n),P)$. In fact, an analogous 
statement holds for all HN types (cf. Proposition \ref{HN strat quot pres}).
\end{rmk}

\subsection{The Hesselink and Shatz stratifications}

In this section, we prove that every Shatz stratum for a HN type $\tau$ is contained 
in a Hesselink stratum of $\quot(V_n \otimes \cO_X(-n),P)$ for $n >\!> 0$; 
this generalises a corresponding result for pure HN types given in \cite{hoskinskirwan}.

\begin{defn}\label{def hess index from HN type} 
For a HN type $\tau = (P_1, \dots , P_s) \in \cH_P$ and natural numbers $(n,m)$, 
we let $\beta_{n,m}(\tau)$ denote the conjugacy class of the rational 1-PS 
\[\begin{array}{rcccc} \lambda_{\beta_{n,m}(\tau)}(t) = \text{diag} &(\underbrace{t^{r_1}, \dots  ,
t^{r_1},} & \dots  & 
\underbrace{t^{r_s}, \dots  , t^{r_s}})& \quad   \text{where} \quad 
r_i:= \frac{P(m)}{P(n)} - \frac{P_i(m)}{P_i(n)} .  
\\ & P_1(n) & & P_s(n) & \end{array}\]
\end{defn}

Since $\tau$ is a HN type, we have $P_1 \succ P_2 \succ \cdots \cdots \succ P_s$ 
and thus, for $m > \!> n > \!> 0$,
\begin{equation}\label{ineq needed} 
\frac{P_1(n)}{P_1(m)}> \frac{P_2(n)}{P_2(m)}  > \cdots \cdots > \frac{P_s(n)}{P_s(m)}
\end{equation}
i.e., the weights $r_i$ are decreasing. 

\begin{thm}[see also \cite{hoskinskirwan}]\label{thm comp strat}
Let $\tau$ be a HN type; then, for $m > \!> n > \! > 0$, 
the Shatz stratum $Q_{n,\tau}$ is a closed subscheme of the Hesselink stratum $S_{\beta_{n,m}(\tau)}$  
in $\quot(V_n \otimes \cO_X(-n),P)$.
\end{thm}
\begin{proof}
If $\tau = (P_1, \dots , P_s)$, we take $n > \! > 0$ so all semistable sheaves with Hilbert polynomial $P_i$ are $n$-regular for $i =1, \dots , s$. Then 
every sheaf with HN type $\tau$ is $n$-regular, as it admits a filtration whose 
successive quotients are $n$-regular, and so can be parametrised by a point in 
the Shatz stratum $Q_{n,\tau}$.

Let $V_n^i := k^{P_i(n)}$ and let 
$Q_{n}^i \subset \quot(V_n^i \otimes \cO_X(-n),P_i)$ be the open subscheme consisting of quotients 
$q : V_n^i \otimes \cO_X(-n) \onto \cE^i$ such that $H^0(q(n))$ is an isomorphism. Let 
$(Q_n^i)^{\pur}$ be the open subscheme of $Q_n^i$ parametrising pure sheaves and let 
$R_n^i$ be the closure of this subscheme in the Quot scheme. By \cite{simpson} Theorem 1.19, 
we can take $m >\!>n$ so that the GIT semistable set for $\SL(V_n^i)$ acting on $R_n^i$ 
with respect to the linearisation $L_{n,m}$ is the lowest Shatz stratum in $Q_n^i$ 
parametrising semistable sheaves; that is,
\[ (R_n^i)^{\SL(V_n^i)-ss}(L_{n,m}) = (Q_n^i)^{ss} := (Q_n^i)_{(P_i)}\]
for $i=1, \dots, s$. Furthermore, we assume $m > \!> n > \!> 0$, 
so that the inequalities (\ref{ineq needed}) hold. 

As described above, the index $\beta = \beta_{n,m}(\tau)$ determines a filtration  
 $0=V^{(0)}_n \subset  \cdots \subset V_n^{(s)}=V_n$ 
where $V_n^i := V_n^{(i)}/V_n^{(i-1)} = k^{P_i(n)}$. 
By construction, the conditions 
(\ref{num cond}) hold for $r(\beta)$ and $l(\beta)$; hence
\[ Z_{\tau,n,m}:= \prod_{i=1}^s (Q_n^i)^{ss}\subset
 \prod_{i=1}^s \quot(V^i_n \otimes \cO_X(-n),P_i)^{\SL(V^i_n)-ss}(L_{n,m}) \subset Z^\lambda_d.\]
Both inclusions are closed inclusions: the first, as $R_n^i$ is closed in 
$\quot(V^i_n \otimes \cO_X(-n),P_i)$. 
 Therefore, $Y_{\tau,n,m} := p_\lambda^{-1}(Z_{\tau,n,m}) \subset Y_d^\lambda$ and 
  $\SL(V_n)Y_{\tau,n,m} \subset S_{\beta}$ are both closed subschemes.
 
To complete the proof, we show that $Q_{n,\tau} = \SL(V_n)Y_{\tau,n,m}$. 
By construction, $ Y_{\tau,n,m}$ 
consists of quotient sheaves $q: V_n \otimes \cO_X(-n) \onto \cF$ with a filtration 
\[ 0 \subset \cF^{(1)} \subset \cdots \subset \cF^{(s)} = \cF \quad 
\text{ where } \quad \cF^{(i)}: = q(V^{(i)}_n \otimes \cO_X(-n))\] 
whose successive quotients $\cF^{(i)}/\cF^{(i-1)}$ are semistable with Hilbert 
polynomial $P_i$; that is, $\cF$ has HN type $\tau$ and thus 
$\SL(V_n^i) Y_{\tau,n,m} \subset Q_{n,\tau}$. 
Conversely, for $q : V_n \otimes \cO_X(-n) \onto \cF$ in $Q_{n,\tau}$ with HN filtration given by
\[0 = \cF^{(0)} \subset \cF^{(1)} \subset \cdots \subset \cF^{(s)}= \cF,\] 
we can choose $g \in \SL(V_n)$ that sends the 
filtration $V^{(i)}_n$ to $W_n^{(i)}:= H^0(q(n))^{-1} (H^0(\cF^{(i)}(n)))$; 
then $g \cdot q \in Y_{\tau,n,m}$ and this finishes the proof. 
\end{proof}

\begin{cor}\label{cor Shatz strata}
Let $\tau$ be a HN type for sheaves on $X$ with Hilbert polynomial $P$. 
For $m>\!> n >\!> 0$, all sheaves with HN type $\tau$ are 
$n$-regular and 
\[Q_{n,\tau}=\SL(V_n)Y_{\tau,n,m}\cong \SL(V_n) \times_{F_{n,\tau}} Y_{\tau,n,m}\]
where 
\begin{enumerate}
\item $F_{n,\tau}=P(\lambda)$ is a parabolic subgroup of $\SL(V_n)$,
\item $Y_{\tau,n,m} = p_\lambda^{-1}(Z_{\tau,n,m})$, and
\item $Z_{\tau,n,m}:= \prod_{i=1}^s (Q_n^i)^{ss}$,
\end{enumerate} 
where the subschemes $(Q_n^i)^{ss} \subset \quot(V_n^i \otimes \cO_X(-n),P_i)$ are as above.
\end{cor}

\subsection{The structure of the HN strata}

We now describe the HN strata $\cC oh^\tau_{P}(X)$.

\begin{prop}\label{HN strat quot pres} 
Let $\tau$ be a HN type. Then, for $n >\!> 0$, we have isomorphisms
 \[\cC oh^\tau_{P}(X) \cong [Q_{n,\tau} /G_n] \cong [Y_{\tau,n,m}/F_{n,\tau}]\]
where $Q_{n,\tau}$ and $Y_{\tau,n,m}$ are locally closed subschemes of 
$\quot( V_n \otimes \cO_X(-n),P)$ and $F_{n,\tau}$ is a parabolic subgroup of 
$G_n = \GL(V_n)$.
\end{prop}
\begin{proof}
 By Corollary \ref{cor Shatz strata}, there exists $n >\!>0$, so all sheaves with HN type
 $\tau$ are $n$-regular and so can be parametrised by a quotient sheaf in the Shatz stratum 
 $Q_{n,\tau} \subset Q_n^o$.
 
The restriction $\cU_{n,\tau}$ of the universal family to $Q_{n,\tau}\times X$ 
has the local universal property for families of sheaves on $X$ of HN type $\tau$ 
by our assumption on $n$. Therefore, we obtain a map
\[Q_{n,\tau} \ra \cC oh^\tau_P(X) \]
that is an atlas for $\cC oh^\tau_P(X)$.  
Two morphisms $f_i : S \ra Q_{n,\tau}$ define isomorphic families 
of sheaves of HN type $\tau$ if and only if they are related by a 
morphism $\varphi : S \ra \GL(V_n)$; 
i.e. $f_1(s) = \varphi(s) \cdot f_2(s)$ for all $s \in S$. 
Hence, the above morphism descends to an isomorphism $[Q_{n,\tau}/G_n] \ra \cC oh^\tau_{P}(X)$. 
The final isomorphism follows from Corollary \ref{cor Shatz strata}.
\end{proof}

\subsection{Stratifications on the stack}\label{sec sheaf strat compare}

It is natural to expect an agreement between the Hesselink and Shatz stratifications, 
as we saw in the case of quiver representations. Further evidence that supports such an 
expectation comes from the agreement between the Yang--Mills 
and HN stratification in gauge theory (cf. \cite{atiyahbott,daskalopoulos}). 
However, we only have a containment 
result: the $\tau$-Shatz stratum is contained in a Hesselink 
stratum for $m$ and $n$ sufficiently large. In this section, we explain 
why these stratifications do not agree and what should be done 
to rectify this. First we note that multiple HN strata can be 
contained in a single Hesselink stratum; the proof follows from 
Definition \ref{def hess index from HN type}.

\begin{lemma}
Let $\tau =(P_1 , \dots , P_s)$ and $\tau' = (P_1', \dots , P_t')$ be 
HN types; then
\[ \beta_{n,m}(\tau) = \beta_{n,m}(\tau') \]
if and only if $s =t$ and, for $i =1, \dots, s$, we have that 
$P_i(n)=P_i'(n)$ and $P_i(m) = P_i'(m)$.
\end{lemma}

\begin{rmk}
We note that if $\dim X \leq 1$, then the assignment $\tau \mapsto \beta_{n,m}(\tau)$ 
is injective for any $m > n$, as the Hilbert polynomial $P$ of any sheaf 
on $X$ has at most degree 1 and so $P$ is uniquely determined by the pair 
$(P(n),P(m))$. 
\end{rmk}

For distinct HN types $\tau \neq \tau'$, we note that 
\[ \beta_{n,m}(\tau) \neq \beta_{n,m}(\tau') \quad \text{for } \: m >\! > n >\!> 0.\]
However, as there are infinite many HN types, we cannot 
pick $m >\! > n >\!> 0$ so that the assignment 
$\tau \mapsto \beta_{n,m}(\tau)$ is injective for all HN types. We can overcome this 
problem by refining the Hesselink strata as suggested in Remark \ref{refine Hess}; 
however, even after such a refinement, the stratifications do not coincide 
(cf. Proposition \ref{relation of strata} and the later comments).

The deeper reason underlying the failure of these stratifications to coincide  
is that the Quot scheme is only a truncated parameter space for sheaves with Hilbert 
polynomial $P$, in the sense that it does not parametrise all sheaves on $X$ 
with Hilbert polynomial $P$. However, the family of all sheaves on $X$ with Hilbert 
polynomial $P$ is unbounded and so there is no scheme that parametrises all 
such sheaves. In fact, we really need to consider the quot schemes $\quot(V_n \otimes \cO_X(-n),P)$ for all $n$ simultaneously.

In the case of vector bundles over a smooth projective curve, Bifet, Ghione 
and Letizia \cite{bgl} constructed an ind-variety of matrix divisors and used a
HN stratification on this ind-variety to derive inductive formulae for the 
Betti numbers. Therefore, one may hope to analogously construct an ind-quot scheme. 
Unfortunately, in our setting, there are no natural maps between these quot 
schemes which allow us to construct such an ind-quot scheme.

The correct space to compare the Hesselink and HN stratifications 
is on the stack $\cC oh_P(X)$. By Corollary \ref{Shatz strat}, 
there is an infinite HN stratification  
 \[\cC oh_{P}(X) = \bigsqcup_{\tau} \cC oh_{P}^\tau(X). \]
We have an increasing open cover of $\cC oh_P(X)$ by the 
substacks $\cC oh_{P}^{n-\reg}(X)$ of $n$-regular sheaves: 
\[ \cC oh_P(X) = \bigcup_{n \geq 0} \cC oh_P^{n-\reg}(X)\]
such that $\cC oh^{n-\reg}_P(X) \cong [Q^{n-\reg}/G_n]$ 
where $Q^{n-\reg}$ is the open subscheme of $Q_n^o$ consisting of 
quotient sheaves $q: V_n \otimes \cO_X(-n) \onto \cE$ such that 
$\cE$ is $n$-regular. In the remainder of this section, by using the 
Hesselink stratification on each Quot scheme $\quot(V_n \otimes \cO_X(-n),P)$, 
we construct a Hesselink stratification on $\cC oh_P(X)$.

For each $n$, we fix $m_n > \!> n$ so the $\SL(V_n)$-linearisation $L_{n,m_n}$ 
on $\quot(V_n \otimes \cO_X(-n),P)$ is ample. Then we have an associated 
Hesselink stratification
\[ \quot(V_n \otimes \cO_X(-n),P) = \bigsqcup_{\beta \in \cB_n} S^n_\beta \]
which we refine, as in Remark \ref{refine Hess}, by tuples $\underline{P} = (P_1, \dots ,P_s)$ of Hilbert polynomials
\[ \quot(V_n \otimes \cO_X(-n),P) = \bigsqcup_{\underline{P} \in \cC_n} S^n_{\underline{P}}.\]
Let $S_{\underline{P}}^{n-\reg}$ be the fibre product of $Q^{n-\reg}$ and $S^n_{\underline{P}}$ in this 
Quot scheme. We have a decomposition
\[ Q^{n-\reg} = \bigsqcup_{\underline{P} \in \cC_n} S_{\underline{P}}^{n-\reg}\]
into finitely many locally closed subschemes. Since $Q^{n-\reg}$ is an open, rather than a closed, subscheme of 
the Quot scheme, the strata $S_{\underline{P}}^{n-\reg}$ do not admit descriptions    
as in Theorem \ref{hess thm proj}. However, as each stratum is $G_n$-invariant, we obtain 
an induced decomposition 
\[ \cC oh^{n-\reg}_P(X) = \bigsqcup_{\underline{P} \in \cC_n} {\cS}_{\underline{P}}^n 
\quad \quad \text{where } \: {\cS}_{\underline{P}}^n  \cong [S_{\underline{P}}^{n-\reg}/G_n]\]
into finitely many disjoint locally closed substacks. 

\begin{prop}\label{relation of strata}
Let $\tau \in \cH_P$ be a HN type. Then, for $n >\! > 0$ and for $n' > \! > n$, 
we have
\[ \cC oh^\tau_P(X) \subset \cS_{\tau}^n \subset 
\bigsqcup_{\tau' \in \cB_\tau^n} \cC oh^{\tau'}_P(X) \subset 
\bigsqcup_{\tau' \in \cB_\tau^n} \cS_{\tau'}^{n'}\]
where $\cB_\tau^n$ is a finite set of HN types $\tau' \geq \tau$.
\end{prop}
\begin{proof}
By Theorem \ref{thm comp strat}, for $n>\!> 0$, the 
Shatz stratum indexed by $\tau$ is contained in the Hesselink stratum indexed by $\beta_{n,m_n}(\tau)$. 
Moreover, $\cC oh^\tau_P(X)$ is contained in the refined Hesselink stratum $\cS^n_\tau$  
indexed by the tuple of polynomials $\tau=(P_1, \dots ,P_s)$. For the second inclusion, 
we note that every quotient sheaf $q : V_n \otimes \cO_X(-n) \onto \cE$ in the  
refined Hesselink stratum $S_\tau^n \subset \quot(V_n \otimes \cO_X(-n),P)$ has a filtration
\[ 0 \subset \cE^{(1)} \subset \cE^{(2)} \subset \cdots \subset \cE^{(s)} =\cE\]
such that $P(\cE^{(i)}/\cE^{(i-1)}) = P_i$; that is $\cE$ has HN type greater 
than or equal to $\tau$ and 
\[ \cS_{\tau}^n \subset \bigsqcup_{\tau' \geq \tau} \cC oh^{\tau'}_P(X). \]
As $S_\tau^{n-\reg}$ has a finite Shatz stratification by HN types, 
we can take a finite index set $\cB_\tau^n$ here. The final inclusion follows 
by applying Theorem \ref{thm comp strat} to the 
finite set of HN types $\cB_\tau^n$.
\end{proof}

It is not the case that for $n' >\!> n$, the $n'$th Hesselink 
stratification refines the $n$th Hesselink stratification. 
For example, we could have two sheaves $\cE$ and $\cE'$ which belong to distinct 
strata $\cS^n_{\underline{P}}$ and $\cS^n_{\underline{P}'}$, but both 
have the same HN type. Moreover, we could have  a stratum  
$\cS_{\underline{P}}^n$ indexed by a tuple of polynomials 
$\underline{P} = (P_1, \dots , P_s)$ which is not a HN type: i.e. we have 
\[ \frac{P_1(n)}{P_1(m_n)} > \cdots > \frac{P_s(n)}{P_s(m_n)} \]
without $P_1 \succ \cdots \succ P_s$; although, then $\underline{P}$ 
will not index a $n'$th Hesselink stratum for $n' >\!> n$.

\begin{lemma}
A tuple of polynomials $\underline{P}$ is a surviving Hesselink index 
(i.e. $\cS^n_{\underline{P}}$ is non-empty for $n >\!> 0$) 
if and only if it is a HN type of a sheaf on $X$ with Hilbert polynomial $P$.
\end{lemma}

We want to define an infinite Hesselink stratification on 
$\cC oh_{{P}}(X)$ such that the infinite strata $\cH ess_\tau$ 
are indexed by the surviving Hesselink 
indices $\tau$. Furthermore, we want this stratification to be a limit over $n$ of the 
finite Hesselink stratifications on the stack of $n$-regular sheaves; more precisely, 
we want a sheaf $\cF$ to be represented by a point of $\cH ess_\tau$ if and only 
if $\cF$ is represented by a point of $\cS^n_\tau$ for $n > \! > 0$.

From this intuitive picture concerning the points we would like the 
infinite Hesselink strata to parametrise, 
it is not clear whether these infinite Hesselink strata 
$\cH ess_\tau$ are locally closed substacks of $\cC oh_{{P}}(X)$. 
To give a more concrete definition, we define the infinite Hesselink 
strata as an inductive limit of stacks. This is slightly delicate as 
there is no reason for the finite Hesselink strata $S^n_{\tau}$ 
to stabilise for large $n$, or even to be comparable. 
For example, we know that there are some sheaves in 
$S^n_{\tau}$ that do not belong to $S^{n'}_{\tau}$ for any $n' >\!> n$; 
however, as $n$ increases there are more and more sheaves that become 
$n$-regular and so can appear in each stratum. 
To circumvent this difficulty, we restrict our attention 
to sufficiently large substacks $\cR^n_{\tau} \subset \cS^n_{\tau}$ 
that form a descending chain. 

\begin{defn}
For each surviving Hesselink index $\tau$, we fix a natural number $N_\tau$ such 
that $\cC oh^{\tau}_P(X) \subset \cS^n_{\tau}$ for all $n \geq N_\tau$.
Then let 
$\cR^{N_\tau}_\tau := \cS_\tau^{N_\tau}$ and
\[ \cR^n_{\tau} := \cR^{n-1}_{\tau} \times_{\cC oh^{n-\reg}_P(X)} \cS^n_{\tau}\]
for $n > N_\tau$. Then we have a descending chain $\cdots \cdots \hookrightarrow \cR^{n}_{\tau} \hookrightarrow 
\cR^{n-1}_{\tau} \hookrightarrow \cdots \cdots \hookrightarrow \cR^{N_{\tau}}_{\tau}$. 
We let $\cH ess_{\tau}$ be the inductive limit of this chain and call this 
the infinite Hesselink stratum for $\tau$.
\end{defn}

\begin{prop}
Let $\tau$ be a surviving Hesselink index; then $\cR^{n'}_{\tau} = \cC oh_P^\tau(X)$ 
for $n' >\! > 0$. In particular, $\cH ess_{\tau} = \cC oh_P^\tau(X)$.
\end{prop}
\begin{proof}
By Proposition \ref{relation of strata}, we have that
\[ \cS^{N_\tau}_\tau \subset \bigsqcup_{\tau' \in \cB_\tau} \cC oh^{\tau'}_P(X)\]
where the index set $\cB_\tau$ is finite and 
$ \cC oh^{\tau'}_P(X) \subset \cS_{\tau'}^{n'}$ for all $\tau' \in \cB_\tau$ 
and $n' >\!> N_\tau$. 
By construction, $ \cC oh^{\tau}_P(X)$ is a substack of $\cR^n_\tau$ and 
$\cR^{n}_\tau$ is a substack of $\cS^n_\tau$ for $n \geq N$. 
We claim that, $\cR_\tau^{n'}$ is a substack of $\cC oh_\tau^{n'}$ for 
$n' >\!> N_\tau$ and, therefore, $\cR_\tau^{n'} = \cC oh^\tau_P(X)$ for 
$n' >\!> N_\tau$.
To prove the claim, we recall that
\[ \cR^{n'}_\tau \subset \cS^{N_\tau}_\tau \subset \bigsqcup_{\tau' \in \cB_\tau}
 \cC oh^{\tau'}_P(X)\]
 and $ \cC oh^{\tau'}_P(X) \subset \cS^{n'}_{\tau'}$, for $n'>\!>N_{\tau}$ and 
 $\tau' \in \cB_\tau$.
  If $\cR^{n'}_\tau$ meets $ \cC oh^{\tau'}_P(X)$ for some $\tau' \neq \tau$,
   then this implies that $\cS^{n'}_\tau$ meets $\cS^{n'}_{\tau'}$, which 
 contradicts the disjointness of these finite Hesselink strata. Therefore, $\cR^{n'}_\tau$ is a substack of $\cC oh_\tau^{n'}$ for $n' >\!> N_\tau$ and this completes the proof.
 \end{proof}
 
\begin{rmk} It is easy to check that a sheaf is represented by a point of 
$\cH ess_\tau$  if and only if it is represented by a point of $\cS^n_\tau$ 
 for $n >\!> 0$.\end{rmk}
 
\begin{cor}\label{cor Hess is HN on stack}
The infinite Hesselink stratification on the stack $\cC oh_P(X)$
\[ \cC oh_P(X) = \bigsqcup_{\tau} \cH ess_{\tau}\]
coincides with the stratification by Harder--Narasimhan types.
\end{cor}

\section{A functorial point of view}

In \cite{ack}, \'{A}lvarez-C\'{o}nsul and King give a functorial construction 
of the moduli space of semistable sheaves on $(X, \cO_X(1))$, 
using a functor
\[ \Phi_{n,m} : \textbf{Coh}(X) \ra \textbf{Reps}({K_{n,m}})\]
from the category of coherent sheaves on $X$ to the category of representations 
of a Kronecker quiver $K_{n,m}$. They prove that, 
for $m >\!> n > \!> 0$, this functor embeds the subcategory of semistable sheaves 
with Hilbert polynomial $P$ into a subcategory of $\theta_{n,m}(P)$-semistable 
quiver representations of dimension $d_{n,m}(P)$ and construct the moduli space 
of semistable sheaves on $X$ with Hilbert polynomial $P$ by using King's construction 
of the moduli spaces of quiver representations. In this section, we consider an 
associated morphism of stacks and describe the relationship between the Hesselink and 
HN strata for sheaves and quivers.

\subsection{Overview of the construction of the functor}

Let $X$ be a projective scheme of finite type over $k$ with very ample 
line bundle $\cO_X(1)$ and let $\textbf{Coh}(X)$ denote the category of 
coherent sheaves on $X$. For natural numbers $m > n$, we let 
$K_{n,m}$ be a Kronecker quiver with vertex set $V:=\{n,m\}$ and 
$\dim H^0(\cO_X(m-n))$ arrows from $n$ to $m$: 
\[ K_{n,m} =  \quad \left( \begin{array}{ccc}  & \longrightarrow & 
\\ n & \vdots & m \\ \bullet & H^0(\cO_X(m-n)) & \bullet \\  & \vdots &  \\ 
& \longrightarrow & \end{array} \right).\]

For natural numbers $m >n$,  we consider the functor 
\[ \Phi_{n,m}:= \text{Hom} (\cO_X(-n) \oplus \cO_X(-m),\:-\:) : \textbf{Coh}(X) \ra \textbf{Reps}({K_{n,m}})\]
that sends a sheaf $\cE$ to the representation $W_\cE$ of $K_{n,m}$ where
\[ W_{\cE,n}:=H^0(\cE(n))  \quad \quad \quad \quad W_{\cE,m}:=H^0(\cE(m)) \]
and the evaluation map $H^0(\cE(n)) \otimes H^0(\cO_X(m-n)) \ra  H^0(\cE(m))$ gives the arrows.

We fix a Hilbert polynomial $P$ and let $\textbf{Coh}_{P}^{n-\reg}(X)$ be the 
subcategory of $\textbf{Coh}(X)$ consisting of $n$-regular sheaves with Hilbert 
polynomial $P$. Then the image of $\Phi_{n,m}$ restricted to $\textbf{Coh}_{P}^{n-\reg}(X)$ 
is contained in the subcategory of quiver representations of dimension vector 
 $d_{n,m}(P)=(P(n),P(m))$. Furthermore, by \cite{ack} Theorem 3.4, if 
 $\cO_X(m-n)$ is regular, then the functor
\[ \Phi_{n,m} : \textbf{Coh}_{P}^{n-\reg}(X) \ra \textbf{Reps}_{d_{n,m}(P)}(K_{n,m})\]
is fully faithful. Let us consider the stability parameter $ \theta_{n,m}(P): = (-P(m),P(n))$ 
for representations of $K_{n,m}$ of dimension $d_{n,m}(P)$. 

\begin{thm}[\cite{ack}, Theorem 5.10]\label{ack thm}
For $m >\!> n > \!> 0$, depending on $X$ and $P$, a sheaf $\cE$ with Hilbert 
polynomial $P$ is semistable if and only if it is pure, $n$-regular and the 
quiver representation $\Phi_{n,m}(\cE)$ is $\theta_{n,m}(P)$-semistable.
\end{thm}

We briefly describe how to pick $m >\!> n > \!> 0$ as required for this theorem to hold; 
for further details, we refer to the reader to the conditions (C1) - (C5) stated in 
\cite{ack} $\S$5.1. First, we take $n$ so all semistable sheaves with Hilbert polynomial 
$P$ are $n$-regular and the Le Potier--Simpson estimates hold. Then we choose $m$ so 
$\cO_X(m-n)$ is regular and, for all $n$-regular sheaves $\cE$ and vector subspaces 
$V' \subset H^0(\cE(n))$, the subsheaf $\cE'$ generated by $V'$ under the evaluation 
map $H^0(\cE(n)) \otimes \cO(-n) \ra \cE$ is $m$-regular. Finally, we take 
$m$ is sufficiently large so a finite list of polynomial inequalities can be determined 
by evaluation at $m$ (see (C5) in \cite{ack}).

We can alternatively consider the functor $\Phi_{n,m}$ as a morphism of stacks. 
We recall we have isomorphisms of stacks 
\[  \cC oh_{P}^{n-\reg}(X)  \cong [Q^{n-\reg}/G_n],\]
where $G_n = \GL(V_n)$, and
\[\cR \text{eps}_{d_{n,m}(P)}(K_{n,m})\cong[\rep_{d_{n,m}(P)}(K_{n,m})/\overline{G}_{d_{n,m}(P)}(K_{n,m})].\]

Let $\cU_n$ be the universal quotient sheaf over $Q^{n-\reg} \times X$ and 
$p : Q^{n-\reg} \times X \ra Q^{n-\reg}$ be the projection map. 
By definition of $Q^{n-\reg}$, we have that $R^ip_* (\cU_n(n)) = 0$ for $ i > 0$;
therefore, by the semi-continuity theorem, $p_*(\cU_n(n))$ 
is a vector bundle over $Q^{n-\reg}$ of rank $P(n)$ and similarly $p_*(\cU_n(m))$ 
is a rank $P(m)$ vector bundle. Hence, by using the evaluation map, we have obtain 
a family of representations of $K_{n,m}$ of dimension $d_{n,m}(P)$ parametrised by 
$Q^{n-\reg}$ that induces a morphism
\[ Q^{n-\reg} \ra \cR \text{eps}_{d_{n,m}(P)}(K_{n,m}).\]
As this morphism is $G_n$-invariant, it descends to a morphism
\[\Phi_{n,m} : \cC oh_{P}^{n-\reg}(X) \ra \cR \text{eps}_{d_{n,m}(P)}(K_{n,m})\]
where we continue to use the notation $\Phi_{n,m}$ to mean the morphism of stacks.

\subsection{Image of the Hesselink strata}

In this section, we study the image of the Hesselink strata under the map
\[\Phi_{n,m} : \cC oh_{P}^{n-\reg}(X) \ra \cR \text{eps}_{d_{n,m}(P)}(K_{n,m}).\]

Let
\[ \quot(V_n \otimes \cO_X(-n),P) = \bigsqcup_{\beta \in \cB_{n,m}} S_\beta\]
be the Hesselink stratification associated to the $\SL(V_n)$-action on this 
Quot scheme with respect to $L_{n,m}$ as we described in $\S$\ref{sec Hess quot}. 
As in $\S$\ref{sec sheaf strat compare}, we consider the induced stratification 
on the stack of $n$-regular sheaves
\begin{equation}\label{hess quot} 
\cC oh^{n-\reg}_P(X) = \bigsqcup_{\beta} \cS_\beta^{n,m}
\end{equation}
where $\cS^{n,m}_\beta = [S^{n-\reg}_\beta/G_n]$ and $S_\beta^{n-\reg} $ is the fibre 
product of $ Q^{n-\reg}$ and $S_\beta$ in this Quot scheme.

We recall that the unstable Hesselink strata are indexed by conjugacy classes of rational 
1-PSs $\lambda_\beta$ of $\SL(V_n)$. Equivalently, the index $\beta$ is 
given by a collection of strictly decreasing rational weights 
$r(\beta) = (r_1, \dots , r_s)$ and multiplicities $l(\beta) = (l_1, \dots , l_s)$ 
satisfying $\sum_{i=1}^s l_i = P(n)$ and $\sum_{i=1}^s r_i l_i = 0$. More precisely, 
we recall that the rational 1-PS associated to $r(\beta)$ and $l(\beta)$ is
\[\lambda_\beta(t)=\diag (t^{r_1},\dots , t^{r_1}, \dots , t^{r_s} ,\dots ,t^{r_s})\]
where $r_i$ appears $l_i$ times.

To define a Hesselink stratification on the stack of representations of $K_{n,m}$ 
of dimension vector $d_{n,m}(P)$, we need to choose a parameter $\alpha \in \NN^2$ 
which defines a norm $|| - ||_\alpha$. 
We choose $\alpha =\alpha_{n,m}(P):=(P(m),P(n))$ due to the following lemma. 

\begin{lemma}\label{how to pick alpha}
Let $\theta=\theta_{n,m}(P)$ and $\alpha=\alpha_{n,m}(P)$. 
Then, for sheaves $\cE $ and $\cF$, we have
\[ \frac{\theta(W_\cE)}{\alpha(W_\cE)} \geq \frac{\theta(W_\cF)}{\alpha(W_\cF)} 
\iff \frac{H^0(\cE(n))}{H^0(\cE(m))} \leq \frac{H^0(\cF(n))}{H^0(\cF(m))}.\]
The same statement holds if we replace these inequalities with strict inequalities.
\end{lemma}
\begin{proof}
By definition of these parameters, we have
\[ \frac{\theta(W_\cE)}{\alpha(W_\cE)}:=
 \frac{-P(m)H^0(\cE(n)) + P(n)H^0(\cE(m))}{P(m)H^0(\cE(n)) + P(n)H^0(\cE(m))}
  = 1 - \frac{2P(m)H^0(\cE(n))}{P(m)H^0(\cE(n)) + P(n)H^0(\cE(m))}.\]
From this, it is easy to check the desired equivalences of inequalities.
\end{proof}

By Theorem \ref{quiver HN is Hess}, we can equivalently 
view the Hesselink stratification (with respect to $\theta$ and $\alpha$) 
as a stratification by HN types:
\[\cR ep_{d_{n,m}(P)}(K_{n,m})=\bigsqcup_\gamma \cR \text{eps}_{d_{n,m}(P)}^{\gamma}(K_{n,m}).\]

\begin{defn}
For an index $\beta$ of the Hesselink stratification (\ref{hess quot}) on the stack 
of $n$-regular sheaves, 
we let $\gamma(\beta):=(d_1(\beta), \dots ,d_s(\beta))$ where
\[ d_i(\beta) = \left(l_i,l_i \frac{P(m)}{P(n)} - l_i r_i\right)\]
and $r(\beta) = (r_1, \dots , r_s)$ and $l(\beta) = (l_1, \dots , l_s)$.
\end{defn}

\begin{lemma}
Let $\beta$ be an index for the Hesselink stratification of $Q^{n-\reg}$; then 
$\gamma(\beta)$ is a HN type for a representation of $K_{n,m}$ of dimension 
$d_{n,m}$ with respect to $\theta$ and $\alpha$.
\end{lemma}
\begin{proof}
Let $r(\beta) = (r_1, \dots , r_s)$ and $l(\beta) = (l_1, \dots , l_s)$ be as above; then  
$r_1 > \dots > r_s$ and 
\[ \sum_{i=1}^s l_i = \dim V_n = P(n) \quad \text{and} \quad \sum_{i=1}^s r_il_i = 0.\]
Let $\gamma(\beta):=(d_1(\beta), \dots ,d_s(\beta))$ be as above; then
\[ \sum_{i=1}^s d_i(\beta)
 =\left(\sum_{i=1}^n l_i, \sum_{i=1}^n l_i \frac{P(m)}{P(n)} - l_i r_i \right)
  = (P(n),P(m)).\]
Since $r_1 > \cdots > r_s$, 
it follows that
\[ \frac{\theta(d_1(\beta))}{\alpha(d_1(\beta))} < \frac{\theta(d_2(\beta))}{\alpha(d_2(\beta))} < 
\cdots \cdots < \frac{\theta(d_s(\beta))}{\alpha(d_s(\beta))}.\]
To complete the proof, we need to verify that 
\[ d_i(\beta) = \left(l_i,l_i \frac{P(m)}{P(n)} - l_i r_i\right) \in \NN^2.\] 
The first number $l_i$ is a multiplicity and so is a natural number, but 
the second number is a priori only rational. 
As $\beta$ is an index for the Hesselink stratification, it indexes 
a non-empty stratum $S_\beta$ and, as this stratum is constructed 
from its associated limit set $Z_d^\lambda$, by Theorem \ref{hess thm proj}, 
it follows that this limit set must also be non-empty. 
Hence this limit set contains a quotient 
sheaf $q : V_n \otimes \cO_X(-n) \ra \oplus_{i=1}^s \cE_i$ such that, by 
Proposition \ref{Zdlambda}, for $i=1, \dots, s$, we have
\[ P(\cE_i,m) = l_i \frac{P(m)}{P(n)} - l_i r_i.\]
Then, as $P(\cE_i,m) \in \NN$, this completes the proof.
\end{proof}

\begin{prop} For a Hesselink index $\beta$, we have 
\[ \Phi_{n,m} \left(\cS^{n,m}_\beta \right) 
\subset  \bigsqcup_{\gamma \geq \gamma(\beta)} \cR \text{eps}^\gamma_{d_{n,m}(P)}(K_{n,m}).\]
\end{prop}
\begin{proof}
Let $q : V_n \otimes \cO_X(-n) \onto \cE$ be a quotient sheaf in $S_\beta^{n-\reg}$ and 
let $r(\beta) = (r_1, \dots , r_s)$ and $l(\beta) = (l_1, \dots , l_s)$ be the associated 
rational weights and multiplicities; thus, $r_1 > \cdots > r_s$ and
\[ \lambda_\beta(t) = \diag (t^{r_1}, \dots , t^{r_1}, \dots , t^{r_s} ,\dots ,t^{r_s}) \]
where $r_i$ appears $l_i$ times. If $\lambda$ is the unique integral primitive 1-PS 
associated to $\lambda_\beta$, then $\beta = ([\lambda],d)$ where $d=||\lambda_\beta||$. 
The 1-PS $\lambda$ induces a filtration 
$0=V^{(0)}_n \subset V^{(1)}_n \subset \cdots \subset V^{(s)}_n = V_n$ 
such that the successive quotients $V^{i}_n$ have dimension $l_i$. 

By Corollary \ref{cor Hess quot}, there exists $g \in \SL(V_n)$ such that we have a filtration
\[ 0 = \cE^{(0)} \subset \cdots \subset \cE^{(i)} := g \cdot q(V^{(i)}_n \otimes \cO_X(-n)) \subset \cdots \subset \cE^{(s)}=\cE\]
where the Hilbert polynomials of $\cE^i := \cE^{(i)}/\cE^{(i-1)}$ satisfy
\[ P(n){P(\cE_i,m)}=l_i({P(m)} - r_iP(n))  \quad \quad \text{for } \: i = 1, \dots ,s. \]

Let $W^{(i)}:= W_{\cE^{(i)}}$ be the quiver representation associated to $\cE^{(i)}$; then we have a filtration
\begin{equation}\label{filtr of W} 
0 = W^{(0)} \subset W^{(1)} \subset \cdots \subset W^{(s)} =W:=W_\cE
\end{equation}
with $ \dim W^{(i)} := (\dim H^0(\cE^{(i)}(n)),\dim H^0(\cE^{(i)}(m)) ) 
= (\dim V_n^{(i)}, P(\cE^{(i)},m) )$, 
due to the fact that $H^0(q(n))$ is an isomorphism and $\cE^{(i)}$ is $m$-regular. 
Let $W_i := W^{(i)}/W^{(i-1)}$; then 
\[\dim W_i=(\dim V_n^i,P(\cE_i,m))=\left(l_i,l_i \frac{P(m)}{P(n)}-l_i r_i \right).\]
As we have a filtration (\ref{filtr of W}) of $W$ whose successive quotients 
have dimension vectors specified by $\gamma(\beta)$, it follows that 
$W_\cE$ has HN type greater than or equal to $\gamma(\beta)$.
\end{proof}

In the above proof, we note that the quiver representation $W_i$ is only isomorphic to 
$W_{\cE_i}$ if $H^{1}(\cE^{(i-1)}(n)) = 0$. In general, this is not the case, 
but it is always the case that $W_i$ is a subrepresentation of $W_{\cE_i}$, 
as $W_{i,n} \subset H^0(\cE_i(n))$ and $W_{i,m} \cong H^0(\cE_i(m))$. 
In particular, it was not possible to use GIT semistability properties of 
the quotient sheaves
$q_i : V_n^i \otimes \cO_X(-n) \ra \cE_i$ 
to deduce $(\theta,\alpha)$-semistability of $W_i$ 
(that is, to show that (\ref{filtr of W}) is the HN filtration of $W$). 
It is not clear to the author if such a result should hold. A more natural way 
to state the above result is the following.

\begin{cor} For a Hesselink index $\beta$, we have
\[ \Phi_{n,m} \left( \bigsqcup_{\beta' \geq \beta}\cS^n_{\beta'} \right) 
\subset  \bigsqcup_{\gamma' \geq \gamma(\beta)} \cR \text{eps}^{\gamma'}_{d_{n,m}(P)}(K_{n,m}).\]
\end{cor}

\subsection{Image of the HN strata}

In this section, we study the image of the HN strata $\cC oh^{\tau}_P(X)$ under the map
\[\Phi_{n,m} : \cC oh_{P}^{n-\reg}(X) \ra \cR \text{eps}_{d_{n,m}(P)}(K_{n,m})\]
for $n$ and $m$ sufficiently large (depending on $\tau$). In fact, we show that a 
HN stratum for sheaves is mapped to a HN stratum for quiver representations. 
Let $\theta=\theta_{n,m}(P): = (-P(m),P(n))$ and 
$\alpha= \alpha_{n,m}(P) := (P(m),P(n))$ be as above.

\begin{defn} 
For a HN type $\tau = (P_1, \dots , P_s) \in \cH_P$ of a sheaf and natural 
 numbers $(n,m)$, we let 
 \[\gamma_{n,m}({\tau}) := (d_{n,m}(P_1), \dots , d_{n,m}(P_s))\] 
 where $d_{n,m}(P_i) = (P_i(n),P_i(m))$.
\end{defn}

As $\tau$ is a HN type of sheaves, we have that ${P_1} \succ {P_2} \succ \cdots \cdots \succ {P_s}$; thus,
\begin{equation}\label{HN type inequal} 
\frac{P_1(n)}{P_1(m)}> \frac{P_2(n)}{P_2(m)}  > \cdots \cdots > \frac{P_s(n)}{P_s(m)}
\quad \quad \text{for } \:  m >\!> n>\!> 0.
\end{equation}
Therefore, by Lemma \ref{how to pick alpha}, for $m >\!> n>\!> 0$, we have
\[ \frac{\theta(d_{n,m}(P_1))}{\alpha(d_{n,m}(P_1))} < \frac{\theta(d_{n,m}(P_2))}{\alpha(d_{n,m}(P_2))} < \cdots \cdots < \frac{\theta(d_{n,m}(P_s))}{\alpha(d_{n,m}(P_s))};\]
i.e., $\gamma_{n,m}(\tau)$ is a HN type for representations of $K_{n,m}$ of dimension $d_{n,m}(P)$ for $m >\!> n>\!> 0$.

\begin{thm}\label{thm ack HN}
Let $\tau = (P_1,\dots ,P_s) \in \cH_P$ be a HN type. Then, for $m >\!> n > \!> 0$, 
 \[\Phi_{n,m} \left( \cC oh_{P}^\tau(X)\right) \subset 
 \cR \text{eps}_{d_{n,m}(P)}^{\gamma_{n,m}({\tau})}(K_{n,m}).\]
\end{thm}
\begin{proof}
We take $m >\!> n > \!> 0$ as needed for Theorem \ref{ack thm} for the 
Hilbert polynomials $P_1, \dots , P_s$. Furthermore, we assume that $m$ and $n$ 
are sufficiently large so the inequalities (\ref{HN type inequal}) hold.

Let $\cE$ be a sheaf on $X$ of HN type $\tau$ and HN filtration given by
\begin{equation}
\label{given filtr} 0 = \cE^{(0)} \subset \cE^{(1)} \subset \cdots \subset \cE^{(s)}= \cE.
\end{equation}
Let $W_\cE$ and $W^{(i)}:=W_{\cE^{(i)}}$ be the associated quiver representations; then we claim that
the induced filtration
\begin{equation}\label{filtr to show} 
0 = W^{(0)} \subset W^{(1)} \subset \cdots \subset W^{(s)} = W_\cE
\end{equation}
is the HN filtration of $ W_\cE$ with respect to $(\theta,\alpha)$ and, moreover, 
that $W_\cE$ has HN type $\gamma_{n,m}(\tau)$. 
Let $\cE_i$ and $W_i$ denote the successive subquotients in the above 
filtrations. Our assumptions on $n$ imply that each $\cE_i$ is $n$-regular and so, 
by induction, 
each $\cE^{(i)}$ is $n$-regular. Therefore, we have exact sequences
\[ 0 \ra H^0(\cE^{(i-1)}(n)) \ra H^0(\cE^{(i)}(n)) \ra H^0(\cE_i(n)) \ra 0\]
that give isomorphisms $W_{\cE_i} \cong W_i$. 
By Theorem \ref{ack thm}, as $\cE_i$ is semistable and $n$-regular with Hilbert 
polynomial $P_i$, the quiver representation $W_i$ is 
$\theta_i:=\theta_{n,m}(P_i)$-semistable. For a subrepresentation 
$W' \subset W_i$, we note that 
\[ \theta_i(W') \geq 0 \iff P_i(n)\dim W'_{v_m}\geq P_i(m)\dim W'_{v_n}  
\iff \frac{\theta(W')}{\alpha(W')} \geq \frac{\theta(W_i)}{\alpha(W_i)}.\]
Therefore, $\theta_i$-semistability of $W_i$ implies $(\theta,\alpha)$-semistability 
of $W_i$. To finish the proof of the claim, i.e. to prove that (\ref{filtr to show}) 
is the HN filtration of $ W_\cE$, it suffices to check that
\[ \frac{\theta(W_1)}{\alpha(W_1)} < \frac{\theta(W_2)}{\alpha(W_2)} 
< \cdots \cdots < \frac{\theta(W_s)}{\alpha(W_s)} \]
or, equivalently, by Lemma \ref{how to pick alpha}, that
\[ \frac{H^0(\cE_1(n))}{H^0(\cE_1(m))} > \frac{H^0(\cE_2(n))}{H^0(\cE_2(m))} 
> \cdots \cdots > \frac{H^0(\cE_s(n))}{H^0(\cE_s(m))}.\]
Since $\cE_i$ is $n$-regular, we have that $H^0(\cE_i(n)) = P_i(n)$ 
and $H^0(\cE_i(m))=P_i(m)$; thus, the above inequalities are equivalent to 
(\ref{HN type inequal}). Moreover, this shows that $ W_\cE$ has HN type 
$\gamma_{n,m}(\tau)$. 
\end{proof}

\begin{rmk}
The assignment $\tau \mapsto \gamma_{n,m}(\tau)$ is not injective 
for exactly the same reason as mentioned in $\S$\ref{sec sheaf strat compare}. 
In fact, more generally, $P \mapsto d_{n,m}(P)$ is not injective, unless 
$ \dim X \leq 1$.
\end{rmk}

\subsection{Adding more vertices}

To determine a polynomial $P(x)$ in one variable of degree at most $d$, it suffices to know the value $P$ takes at $d+1$ different values of $x$. In this final section, by using this observation, we generalise the construction of \'{A}lvarez-C\'{o}nsul and King by adding more vertices so that the map $\tau \mapsto \gamma(\tau)$ is injective. 

For a tuple $\underline{n}=(n_0, \cdots , n_d)$ of increasing natural numbers, we define a functor 
\[ \Phi_{\underline{n}}:=  \text{Hom} (\bigoplus_{i=0}^d\cO_X(-n_i),\:-\:) : \cC oh(X) \ra \cR \text{eps}(K_{\underline{n}})\]
where $K_{\underline{n}}$ denotes the quiver with vertex set $V=\{n_0, \dots , n_d\}$ and $\dim H^0(\cO_X(n_{i+1} - n_i))$ arrows from $n_i$ to $n_{i+1}$: 
\[ K_{\underline{n}} =  \quad \left( \begin{array}{ccccccc}  & \longrightarrow & & & & \longrightarrow &
\\ n_0 & \vdots & n_1 & & n_{d-1} & \vdots & n_d \\ \bullet & H^0(\cO_X(n_1-n_0)) & \bullet & \quad  \cdots \quad &  \bullet & H^0(\cO_X(n_d -n_{d-1})) & \bullet \\  & \vdots &    \\ 
& \longrightarrow & &  & & \longrightarrow & \end{array} \right).\]
More precisely, if $\cE$ is a coherent sheaf on $X$, then $\Phi_{\underline{n}}(\cE)$ is the quiver representation denoted $W_{\cE}=(W_{\cE,n_0}, \dots , W_{\cE,n_d}, \text{ev}_1, \dots , \text{ev}_d)$ where we let
$W_{\cE,l} := H^0(\cE(l))$ and define the maps $\text{ev}_l : H^0(\cE(n_{l-1})) \otimes H^0(\cE(n_l - n_{l-i})) \ra H^0(\cE(n_i))$ using the evaluation map on sections. 

For the dimension vector $d_{\underline{n}}(P):=(P(n_0), \dots , P(n_d))$, we observe that
\[\Phi_{\underline{n}}: \cC oh_{P}^{n_0-\reg}(X) \ra \cR \text{eps}_{d_{\underline{n}}(P)}(K_{\underline{n}}).\]
Moreover, we note that the assignment that maps a Hilbert polynomial $P$ of a sheaf on $X$ to the corresponding dimension vector $d_{\underline{n}}(P)$ is injective.

We want to choose a stability parameter $\theta \in \ZZ^{d+1}$ 
for representations of $K_{\underline{n}}$ of dimension vector 
$d_{\underline{n}}(P)$ such that $\Phi_{\underline{n}}$ sends semistable 
sheaves to $\theta$-semistable quiver representations for $\underline{n} > \! > 0$ 
(that is, for $n_d > \! > n_{d-1} > \! > \cdots > \! > n_0 > \! > 0$). Let
\[ \theta= \theta_{\underline{n}}(P):= (\theta_0, \dots , \theta_d) \quad \text{where} \quad \theta_i:= \sum_{j <i} P(n_j) - \sum_{j>i} P(n_j); \]
then $\sum_{i=0}^d \theta_i P(n_i) = 0$. 
The following lemma demonstrates this is a suitable choice.

\begin{lemma}\label{correct theta}
Let $\theta = \theta_{\underline{n}}(P)$ as above and $\cE$ be a sheaf on $X$. 
If, for all $j < i$, we have 
\[ \frac{h^0(\cE(n_j))}{ h^0(\cE(n_i))}
 \leq \frac{P(n_j)}{P(n_i)},\]
then $\theta(\cE) \geq 0$.
\end{lemma}
\begin{proof}
From the definition 
$\theta (\cE) := \sum_{i=0}^d \theta_i h^0(\cE(n_i))$, it follows that 
\begin{align*}
\theta (\cE) & = \sum_{i} \left( \sum_{j< i} P(n_j)  - \sum_{j>i} P(n_j) \right) h^0(\cE(n_i)) \\
& =\sum_i \sum_{j<i} \left( P(n_j)  h^0(\cE(n_i)) - P(n_i) h^0(\cE(n_j) \right) \geq 0
\end{align*}
by using the given inequalities.
\end{proof}

Using this lemma, it is easy to prove the following corollary
analogously to \cite{ack} Theorem 5.10.

\begin{cor}
Let $P$ be a fixed Hilbert polynomial of a sheaf on $X$; 
then for $\underline{n} >\! > 0$, we have
\[ \Phi_{\underline{n}} (\cC oh_{P}^{\text{ss}}(X) ) \subset \cR \text{eps}_{d_{\underline{n}}(P)}^{\theta_{\underline{n}}(P)}(K_{\underline{n}}).\]
\end{cor}

Finally, we need to find the value of the parameter $\alpha \in \NN^{d+1}$ 
needed to define the correct notion of HN filtrations of representations of 
$K_{\underline{n}}$ in order to prove an analogue of Theorem \ref{thm ack HN}. 
The following lemma, whose proof is analogous to Lemma \ref{correct theta}, shows that 
\[ \alpha= \alpha_{\underline{n}}(P):= (\alpha_0, \dots , \alpha_d) \quad \text{where} \quad \alpha_i:= \sum_{j <i} P(n_j) + \sum_{j>i} P(n_j) \]
is a suitable choice.

\begin{lemma}
Let $\theta = \theta_{\underline{n}}(P)$ and $\alpha= \alpha_{\underline{n}}(P)$ be as above and $\cE$ and $\cF$ be two sheaves on $X$. If, for all $j < i$, we have
\[\frac{h^0(\cE(n_j))}{ h^0(\cE(n_i))} \leq \frac{h^0(\cF(n_j))}{h^0(\cF(n_i))}, \]
then 
\[ \frac{\theta(\cE)}{\alpha(\cE)} \geq \frac{\theta(\cF)}{\alpha(\cF)}.\]
\end{lemma}

Finally, we deduce the required corollary in the same way as the proof of Theorem \ref{thm ack HN}, by using the above lemma.

\begin{cor}
Fix a Hilbert polynomial $P$ and let $\theta = \theta_{\underline{n}}(P)$ and $\alpha= \alpha_{\underline{n}}(P)$. For a HN type $\tau=(P_1, \dots , P_s)$ of sheaves on $X$ with Hilbert polynomial $P$, we have, for $\underline{n} > \! > 0$, that
\[ \Phi_{\underline{n}} (\cC oh_{P}^{\tau}(X) ) \subset \cR \text{eps}_{d_{\underline{n}}(P)}^{\gamma_{\underline{n}}(\tau)}(K_{\underline{n}})\]
where
\[ \gamma_{\underline{n}}(\tau): = (d_{\underline{n}}(P_1), \dots , d_{\underline{n}}(P_s)) \]
and the HN stratification on the stack of quiver representations is taken with respect to $(\theta, \alpha)$. 
In particular, the assignments $P \mapsto d_{\underline{n}}(P)$ and $\tau \mapsto \gamma_{\underline{n}}(\tau)$ are both injective.
\end{cor}

\bibliographystyle{amsplain}
\bibliography{references}

\medskip \medskip

\noindent{Freie Universit\"{a}t Berlin, Arnimallee 3, Raum 011, 14195 Berlin, Germany} 

\medskip \noindent{\texttt{hoskins@math.fu-berlin.de}}

\end{document}